\documentclass[11pt]{article}
\usepackage[top=1in, bottom=1in, left=1in, right=1in]{geometry}
\usepackage{mathtools}
\usepackage{footnotebackref}
\usepackage{xcolor}

\usepackage{cleveref}
\usepackage{amsmath,amssymb,amsfonts}
\usepackage{amsthm}
\usepackage{url}
\usepackage{algorithm,algorithmic} 
\usepackage{mathtools}
\usepackage{tcolorbox}
\usepackage{bbm}
\usepackage{todonotes}
\usepackage{caption}
\usepackage{subcaption}
\usepackage{graphicx}
\usepackage{thm-restate}

\newcommand{\real}{\mathbb{R}}

\newcommand{\RR}{\mathbb{R}}
\newcommand{\R}{\mathbb{R}}
\newcommand{\symMat}{\mathbb{S}}

\usepackage{accents}

\mathchardef\mhyphen="2D 

\DeclareMathOperator{\rank}{rank}
\DeclareMathOperator{\nullspace}{nullspace}
\DeclareMathOperator{\diag}{diag}
\DeclareMathOperator{\range}{range}

\newtheorem{theorem}{Theorem}[section]
\newtheorem{lemma}{Lemma}[section]
\newtheorem{definition}{Definition}
\theoremstyle{plain}
\newtheorem{exm}{Example}[section]
\newtheorem{remark}{Remark}[section]

\newcommand{\Amap}{\mathcal{A}}

\newcommand{\dm}{d}

\newcommand{\inprod}[2]{\langle #1, #2 \rangle}
\newcommand{\twonorm}[1]{\left\|#1\right\|_2}

\newcommand{\fronorm}[1]{\left\|#1\right\|_{\mbox{\tiny{F}}}}
\newcommand{\fronormA}[1]{\|#1\|_{\mbox{\tiny{F}}}}

\newcommand{\opnorm}[1]{\left\|#1\right\|_{\mbox{\tiny{\textup{op}}}}}
\newcommand{\nucnorm}[1]{\left\|#1\right\|_*}

\newcommand{\tr}{\mathop{\bf tr}}

\let\originalleft\left
\let\originalright\right
\renewcommand{\left}{\mathopen{}\mathclose\bgroup\originalleft}
\renewcommand{\right}{\aftergroup\egroup\originalright}

\usepackage{multirow}

\newcommand{\cL}{{\cal L}}
\newcommand{\cLssv}{{\cal L}_{\text{ssv}}}

\providecommand{\keywords}[1]{\textbf{\textit{Keywords---}} #1}

\title{On Squared-Variable Formulations for Nonlinear Semidefinite programming}

\author{Lijun Ding\thanks{Department of Mathematics, University of California San Diego, La Jolla, CA 92093, USA (\texttt{l2ding@ucsd.edu})}     
\and 
Stephen J. Wright\thanks{Department of Computer Sciences,
University of Wisconsin, Madison, WI 53706, USA (\texttt{swright@cs.wisc.edu}).
}}

\begin{document}
\maketitle

\begin{abstract}
    In optimization problems involving smooth functions and real and matrix variables, that contain matrix semidefiniteness constraints, consider the following change of variables: Replace the  positive semidefinite matrix $X \in \symMat^\dm$ by a matrix product $FF^T$, where $F \in \real^{\dm \times \dm}$ or $F \in \symMat^\dm$.
    The formulation obtained in this way is termed ``squared variable," by analogy with a similar idea that has been proposed for real (scalar) variables.
    It is well known that points satisfying first-order conditions for the squared-variable reformulation do not necessarily yield first-order points for the original problem.
    There are closer correspondences between second-order points for the squared-variable reformulation and the original formulation. 
    These are explored in this paper, along with correspondences between local minimizers of the two formulations.
\end{abstract}


\section{Introduction} \label{sec:intro}

Consider a smooth objective $f:\real^n \rightarrow \real$ and a smooth matrix function $C: \real^n \rightarrow \symMat^\dm$, where $\symMat^\dm$ denotes the space of symmetric $\dm  \times \dm$ matrices.
From these functions, we define a positive semidefinite (PSD) constrained optimization problem
\begin{equation}
    \label{eq:f} \tag{NSDP}
    \min_{x \in \real^n} \, f(x) \;\; \mbox{s.t.} \;\; C(x) \succeq 0,
\end{equation}
with ``$\succeq 0$" denoting positive semidefiniteness.\footnote{A symmetric matrix is positive semidefinite if  all its eigenvalues are nonnegative.}
An equivalent formulation with matrix squared-slack variables (SSV) is
\begin{equation}
    \label{eq:ssv} \tag{SSV}
    \min_{x \in \real^n,F\in \R^{\dm \times\dm}} \, f(x) \;\; \mbox{s.t.} \;\; C(x) = FF^\top.
\end{equation}
(Note that $F$ is not assumed to be symmetric in general.)
For problems whose only constraint is a semidefiniteness constraint on the matrix variable, a direct-substitution approach can be used instead. 
In particular, the problem 
\begin{equation}
    \label{eq:bc} \tag{BC}
    \min_{X \in \symMat^\dm}\, h(X) \;\; \mbox{s.t.} \;\; X \succeq  0,
\end{equation}
with a smooth objective $h:\symMat^\dm \rightarrow \RR$, can be reformulated as an  unconstrained problem
\begin{equation} \label{eq:bc.dss} \tag{DSS}
\min_{F \in \RR^{\dm \times \dm}} \, g(F) := h(FF^\top ).
\end{equation}

We call \eqref{eq:ssv} and \eqref{eq:bc.dss}  ``squared-variable formulations" to emphasize the connection with similar formulations involving scalar variables, although the matrix $F$ is not actually ``squared" in these formulations unless it happens to be symmetric.

This paper studies the correspondence between \eqref{eq:f} and \eqref{eq:ssv} of points satisfying second-order necessary conditions, and the similar correspondence between \eqref{eq:bc} and \eqref{eq:bc.dss}. 
We use the terminology ``2NC" to mean ``second-order necessary conditions for optimality," ``2NP" to mean  ``a point satisfying 2NC," ``2SC" to mean ``second-order sufficient conditions for optimality," ``1C" to mean ``first-order optimality conditions," and  ``1P" to mean ``a point satisfying 1C."

2NC for \eqref{eq:ssv} and \eqref{eq:bc.dss} are well known from the optimality theory for  nonlinear programs; see for example \cite[Section~12.4]{NocW06}. 
However, for problems \eqref{eq:f} and \eqref{eq:bc} with PSD constraints, 2NC are less well known and are more complicated due to the curvature of the PSD cone.
\footnote{Strictly speaking, our notion of 2NC in Section \ref{sec: DS} and \ref{sec: ssv} is a weak form of second-order condition as it requires the key inequality to hold only over a \emph{subspace}. A stronger notion of second-order necessary condition, termed s2NC, requires this inequality to hold over a \emph{cone} that contains the subspace. However, verifying s2NC can be computationally intractable. It reduces to 2NC under the so-called strict complementarity condition, which is almost always assumed in the literature \cite{shapiro1997first,lourencco2018optimality}, though not here. 
Since our ``weaker" 2NC is tractable to check and is tied closely to the squared-slack variable formulation, it is the central notion of our paper. We discuss s2NC and its relationship with 2NC in Appendix \ref{sec: SOC discussion}.}
We define these 2NC in Sections~\ref{sec: DS} and \ref{sec: ssv} and discuss their history in Section~\ref{sec: RLwork}. 

\paragraph{Main result} Despite the complicated nature of the PSD cone, there is a close relationship between 2NP for the original  constrained formulations and the squared-variable formulations.
The main result of this paper, Theorem~\ref{th:f}, shows that 
\begin{quote}
    There is a direct correspondence between 2NP of \eqref{eq:f} and 2NP of \eqref{eq:ssv}.
    That is, if $F$ is a 2NP of \eqref{eq:ssv}, then $X=FF^\top$ is a 2NP of \eqref{eq:f}. 
    Conversely,
    if $X$ is a 2NP of \eqref{eq:f}, then any $F$ with $FF^\top = X$ is a 2NP of \eqref{eq:ssv}.
\end{quote}
The corresponding conclusion holds for \eqref{eq:bc} and \eqref{eq:bc.dss}, as shown in Theorem~\ref{th:bc}.

The correspondence between 2NPs of the respective formulations leads to potential computational benefits. 
Since certain algorithms (in particular, algorithms for equality constrained nonlinear programming) can be applied to \eqref{eq:ssv} but not to \eqref{eq:f}, our main result opens the way to a wider range of approaches for solving \eqref{eq:f}.
{Publicly available solvers for the formulation \eqref{eq:f} are scarce, except in special cases (for example, when $f$ and $C$ are linear functions of $x$).}
Similarly,  the unconstrained problem \eqref{eq:bc.dss} admits a wide and diverse class of methods by comparison with the semidefinite-constrained matrix optimization problem \eqref{eq:bc}.
Achieving a 2NP is normal in practice for first-order methods in unconstrained optimization \cite{lee2016gradient} and is provable for many algorithms under appropriate conditions \cite{jin2017escape,sahin2019inexact,xie2021complexity}. 

Our result can be regarded as a matrix analog of the squared variable approach in nonlinear programming, \cite[Section~3.3.2]{Ber99}, \cite[Theorem 1]{li2021simplex}, and \cite[Theorem 2.3 and Theorem 3.3]{ding2023squared}. Compared to these earlier results, two important technicalities appear. 
First, due to the positive semidefiniteness constraint,  there is an extra curvature term in the second-order condition of \eqref{eq:f}. Second, the noncommutativity of matrices requires delicate handling.

\underline{Local Minimizers.}
The correspondence between local minimizers (rather than 2NPs) of the two formulations \eqref{eq:f} and \eqref{eq:ssv} is as follows. 
Theorem~\ref{thm:f-ssv-local_min} shows that $(x,F)$ is a local minimizer of \eqref{eq:ssv} if and only if $x$ is a local minimizer of \eqref{eq:f}.  
However, the equivalence fails to hold for \emph{strict} local minimizers, due to rotational symmetry in the product $FF^\top$. 
(That is, we can replace $F$ by $FQ$ in \eqref{eq:ssv}, where $Q$ is any orthogonal $d \times d$ matrix, without affecting the product $FF^T$.)
For the same reason, the second-order sufficient condition for \eqref{eq:ssv} never holds except in the special case $F=0$.

\paragraph{A Variant: Symmetric $F$} 
A variant on the formulations above can be obtained by requiring the matrix $F$ in  \eqref{eq:ssv} and \eqref{eq:bc.dss} to be symmetric, leading to the following variants of  \eqref{eq:ssv} and \eqref{eq:bc.dss}, respectively:
\begin{equation}
    \label{eq:ssv-sym} \tag{SSV-Sym}
    \min_{x \in \real^n,F\in \symMat^\dm} \, f(x) \;\; \mbox{s.t.} \;\; C(x) = F^2,
\end{equation} 
\begin{equation} \label{eq:bc.ssv-sym} \tag{DSS-Sym}
\min_{F \in \symMat^\dm} \, g(F) := h(F^2).
\end{equation}
There are several differences between the symmetrized version \eqref{eq:ssv-sym} and the original version \eqref{eq:ssv} in terms of local minimizers and 2NPs. 
 


\underline{Local Minimizers.} 
Theorem \ref{thm:f-ssv-sym-local_min} says if $x$ is a (strict) local minimizer of \eqref{eq:f}, then any $(x,F)$ with $F$ symmetric and $C(x) =F^2$  is a (strict) local minimizer of  \eqref{eq:ssv-sym}. 
The converse is true for strict local minimizers as well. 
However, for the statement to hold for local minimizers (without strictness), it requires the condition that no pair of nonzero eigenvalues of the matrix $F$ sums to $0$, {that is, 
\begin{equation} \label{eq:ec} \tag{EC}
    \sigma_{i} + \sigma_{j} \not =0, \quad \mbox{for any two nonzero eigenvalues $\sigma_i$ and $\sigma_j$ of $F \in \symMat^\dm$.}
\end{equation}
{We refer to this condition in subsequent discussions as the ``eigenvalue condition." This condition can be traced back to the study of Euclidean Jordan algebra \cite[Theorem IV.2.1 and Corollary IV.2.6]{faraut1994analysis}, \cite[Proposition 1]{sturm2000similarity}, and is also recently revisited by \cite[Proposition 1]{lourencco2018optimality} in the context of nonlinear semidefinite programming.} 
The eigenvalue condition is also necessary in the sense that there exist $f$, $C$, $x$, and $F$ in \eqref{eq:ssv-sym} such that $(x,F)$ is a local minimizer of \eqref{eq:ssv-sym}  but $x$ is not a local minimizer of \eqref{eq:f}; see Example \ref{exm: bc.ssv-sym-local}.\footnote{Our theorem and example correct a misunderstanding in the literature \cite[p.2, under (P2)]{lourencco2018optimality} that the equivalence holds without the eigenvalue condition \eqref{eq:ec}.}
%

%


\underline{Second-Order Necessary Points.} 
Theorem \ref{th:ssv.sym} says that if $x$ is a 2NP of \eqref{eq:f}, then any $(x,F)$ satisfying $C(x) =F^2$ and $F$ symmetric is a 2NP of \eqref{eq:ssv-sym}. 
However, even if $(x,F)$ is a 2NP of \eqref{eq:ssv-sym}, the point $x$ may not even be a 1P of \eqref{eq:f}, {unless $F$ satisfies the eigenvalue condition \eqref{eq:ec}}
(see Example \ref{exm:ssv-sym}). 
This situation can be resolved by requiring $F$ to be PSD.
However, this additional requirement would remove the main algorithmic advantage of \eqref{eq:ssv-sym}, since maintaining  the PSD property is computationally difficult.
This conclusion contrasts starkly with our main result (Theorem~\ref{th:f})  that 2NPs of \eqref{eq:f} and \eqref{eq:ssv} are equivalent. 

The statements above  hold for \eqref{eq:bc}, \eqref{eq:bc.dss}, and \eqref{eq:bc.ssv-sym}, as shown in Theorems~\ref{th:bc}, \ref{th:bc-ssv.sym}, \ref{thm:bc-ssv-local_min}, and \ref{thm:bc-ssv-sym-local_min}, with one exception: 
If $F$ is a 2NP of \eqref{eq:bc.ssv-sym}, then it is a 1P of \eqref{eq:bc}, but not necessarily a 2NP of \eqref{eq:bc}.
The latter point is illustrated in Example~\ref{exm: bc.ssv-sym}.

\subsection{Related work}\label{sec: RLwork}

We discuss here other works that consider product-form variable substitutions and second-order optimality conditions for matrix optimization problems.

\paragraph{Burer-Monteiro} 
Our squared-variable reformulation \eqref{eq:bc.dss} makes use of a {\em square} matrix $F$. 
By contrast, the celebrated Burer-Monteiro approach \cite{burer2003nonlinear} applied to \eqref{eq:bc} allows $F$ to be non-square, of dimensions $d\times k$ with $k\leq d$, leading to the formulation 
\begin{equation} \label{eq: bc.ssv.bm} \tag{DSS-BM}
\min_{F\in \R^{d\times k}} \, g(F) := h(FF^\top ).
\end{equation}
{For the much-studied linear semidefinite program (a convex problem) defined by}
\begin{equation}\label{eq: sdp} \tag{LSDP}
\min_{X\in \symMat^d} \; \tr(CX) \quad \text{s.t.}\; \tr(A_iX) =b_i, \, i=1,\dots,m, \;\text{and}\;X\succeq 0,     
\end{equation}
where $A_i,C\in \symMat^d$ and $b_i \in \R$, the  Burer-Monteiro formulation is
\begin{equation}    \label{eq: sdp.bm} \tag{LSDP-BM}
\min_{F\in \R^{\dm \times k}} \; \tr(CFF^\top ) \quad \text{s.t.}\; \tr(A_i FF^\top ) =b_i, \, i=1,\dots,m.
\end{equation}
From \cite[Proposition 2.3]{burer2003nonlinear}, it is known that for $X = FF^\top$ with $F\in \R^{d\times d}$ (that is, $k=d$ and $F$ square), the matrix $X$ is a local minimizer of \eqref{eq: sdp} if and only if $F$ is a local minimizer of \eqref{eq: sdp.bm}. However, finding a local minimizer for a nonconvex problem
is in general NP-hard. 
From \cite[Theorem~4.1]{burer2005local} and \cite[Lemma~10]{bandeira2016low}, a matrix $F$ that is a 2NP of 
\eqref{eq: sdp.bm} yields a 1P for \eqref{eq: sdp} if  the matrix $F$ is \emph{rank deficient}, that is, $\rank(F) <k$. 
The same property holds for \eqref{eq:bc} and \eqref{eq:bc.dss} for convex functions $h$, from \cite[Lemma~1]{bhojanapalli2018smoothed}. \footnote{Those works actually consider only convex problems and  address global optimality but not first-order conditions directly. However, the approach used in those papers is to check whether or not first-order conditions are satisfied.}
\footnote{Note this result does not show that a 2NP of \eqref{eq:bc.dss} yields a 1P for \eqref{eq:bc}, unless $F$ is rank deficient. By contrast, our result of Theorem~\ref{th:bc} shows that if $F$ a 2NP of \eqref{eq:bc.dss} then it yields a 2NP of \eqref{eq:bc}, even if $F$ has full rank $\dm$. } 
Much work \cite{barvinok1995problems,pataki1998rank,boumal2016non,boumal2020deterministic,bhojanapalli2018smoothed,waldspurger2020rank,zhang2022improved,cifuentes2022polynomial,o2022burer} has been devoted to understanding how the choice of $k$ affects the {\em global optimality} of 2NPs of \eqref{eq: bc.ssv.bm} with convex $h$, or of 2NPs for \eqref{eq: sdp.bm}. 
However, as shown in \cite[Section 3]{bhojanapalli2018smoothed}, the parameter $k$ may need to be as large as $n-1$, even when the optimal solution is rank $1$. 
Otherwise, there are points that satisfy second-order conditions yet are not globally optimal. 
Even for such structured problems as Max-Cut SDP, $k$ must be larger than $n/2$ to prevent such a situation \cite{o2022burer}. 
Regarding the negative results, we refer the reader to \cite[Section~1.2]{bhojanapalli2018smoothed}, \cite[pp.~2-3]{waldspurger2020rank}, and \cite[p.~2, Appendix~A]{o2022burer},  for detailed discussions on the choice of $k$ and the global optimality of points satisfying second-order conditions for  \eqref{eq: bc.ssv.bm} or of \eqref{eq: sdp.bm}.

\paragraph{Second order conditions} As discussed in \cite{shapiro1997first,forsgren2000optimality,lourencco2018optimality}, the second order condition of \eqref{eq:f} involves an extra term that arises from the curvature of the cone of positive semidefinite matrices. 
In \cite[Theorem 9]{shapiro1997first}, under the transversality condition \cite[Definition 4]{shapiro1997first} 
and a strict complementarity condition \cite[(18)]{shapiro1997first}, Shapiro derived second-order sufficient conditions together with necessary conditions for local minimizers based on \cite[Theorem~4.2]{cominetti1990metric}, by making use of first-order and second-order tangent cones. 
{Like the linear independence constraint qualification (LICQ) for nonlinear programming problems in their more familiar form,} the transversality condition ensures the existence and uniqueness of the Lagrangian multiplier \cite[Proposition~7]{shapiro1997first}.
But this condition is not equivalent to LICQ even if $C(x)$ in \eqref{eq:f} is  diagonal for all $x$ \cite[p.~309]{shapiro1997first}. 
In \cite[Theorems~2, 3]{forsgren2000optimality}, Forsgren derives second-order conditions for local minimizers by following a more elementary approach, making use of  eigenvalue decomposition, the matrix Farkas Lemma, and classical techniques from nonlinear programming \cite[Chapter~2]{fiacco1990nonlinear} and \cite[Chapter~15]{luenberger1984linear}. 
Notably, the constraint qualification \cite[Section 2.3]{forsgren2000optimality} reduces to the usual LICQ in nonlinear programming when the matrix $C(x)$ in \eqref{eq:f} is diagonal for all $x$. 
%
%
%
%
In \cite{lourencco2018optimality},  under the transversality condition and the strict complementarity condition, Louren{\c c}o et al. derived second-order conditions for \eqref{eq:f} using the more approachable second-order conditions for \eqref{eq:ssv-sym} (but not \eqref{eq:ssv}).  
We note that \cite{shapiro1997first,lourencco2018optimality} actually considered a stronger form of 2NC, which we term ``s2NC," while \cite{forsgren2000optimality}, for reasons of computational feasibility,  considered the same (weaker) 2NC that we consider here.
Under strict complementarity, which  is assumed throughout \cite{lourencco2018optimality}, s2NC reduces to  2NC. 
We discuss the relationship between s2NC and 2NC further in Appendix~\ref{sec: SOC discussion}.


\paragraph{Second-order points of \eqref{eq:ssv-sym} and \eqref{eq:f}}
Louren{\c c}o et al. \cite{lourencco2018optimality} consider the relationship of the points satisfying second-order sufficient conditions (2SC)  or second-order necessary conditions (2NC) for \eqref{eq:ssv-sym}  and \eqref{eq:f}, and derive second-order conditions for \eqref{eq:f} from those conditions of \eqref{eq:ssv-sym}, making the 2SC and 2NC of \eqref{eq:f} more approachable. 
For {\em sufficient} conditions, they show the following: If $(x,\Lambda)$ (where $\Lambda$ is a Lagrangian multiplier for the constraint $C(x) \succeq 0$) satisfies 2SC of \eqref{eq:f} with strict complementarity, then $(x,\sqrt{C(x)},\Lambda)$ satisfies 2SC for \eqref{eq:ssv-sym}. \footnote{Here $\sqrt{C(x)}$ is defined to be the square root of $C(x)$ formed by taking the nonnegative square roots of the eigenvalues. 
That is, $\sqrt{C(x)}$ is positive semidefinite and $C(x) = \sqrt{C(x)}\sqrt{C(x)}$.}
Conversely, if $(x,F,\Lambda)$ satisfies 2SC of \eqref{eq:ssv-sym}, then under strict complementarity and transversality of $(x,\Lambda)$, it satisfies 2SC for \eqref{eq:f}.  
For {\em necessary} conditions, they show the following: Under transversality and strict complementarity, $(x,\Lambda)$ satisfies 2NC of \eqref{eq:f} if and only if $(x,\sqrt{C(x)},\Lambda)$ satisfies 2NC of \eqref{eq:ssv-sym} and $\Lambda \succeq 0$. Thus, 2SC for a point $(x,\Lambda)$ satisfying transversality and strict complementarity can be written in terms of 2SC of \eqref{eq:ssv-sym}. Similarly, 2NC for a point $(x,\Lambda)$ satisfying transversality and strict complementarity can be written in terms of 2NC of \eqref{eq:ssv-sym} with the extra condition $\Lambda \succeq 0$.

Compared to the equivalence of 2NC results in \cite{lourencco2018optimality} for \eqref{eq:bc} and \eqref{eq:bc.ssv-sym}, our results show that the transversality condition and the strict complementarity condition are not necessary.
We also show the equivalence result for 2NP over a larger set of points. 
Specifically, in Theorem \ref{th:ssv.sym}, we show that {under the eigenvalue condition \eqref{eq:ec} (that no pair of nonzero eigenvalues of $F$ sums to zero),}
a point $(x,\Lambda)$ satisfies 2NC of \eqref{eq:f} if and only if $(x,F,\Lambda)$ satisfies 2NC of \eqref{eq:ssv-sym}. 
(In this case,  we do not require $(x,F,\Lambda)$ in \eqref{eq:ssv-sym} to satisfy $F=\sqrt{C(x)}$ and $\Lambda \succeq 0$, which is required in \cite{lourencco2018optimality}.) 
We also provide an example confirming the necessity of the eigenvalue condition \eqref{eq:ec}:
In Example~\ref{exm:ssv-sym}, we show for some choices of $f$ and $C$, there is a point $(x,F,\Lambda)$ satisfying  2NC for \eqref{eq:ssv-sym} with $F$ violating \eqref{eq:ec}, such  $(x,\Lambda)$ does not satisfy even first-order conditions for \eqref{eq:f}.
Moreover, in Theorem~\ref{th:bc-ssv.sym}, we show that in the special case \eqref{eq:bc.ssv-sym}, if $F$ satisfies the 2NC of \eqref{eq:ssv-sym}, then $F^2$ is a 1P of \eqref{eq:bc}.  
In Example~\ref{exm: bc.ssv-sym}, we confirm the necessity of this condition:
We show that for some choice of function $h$, there is a matrix $F$, {feasible in \eqref{eq:bc.ssv-sym} and satisfying 2NC for \eqref{eq:bc.ssv-sym} but violating  the eigenvalue condition \eqref{eq:ec}}, such that $F^2$ is only a 1P of \eqref{eq:bc} and not a 2NP.

Finally, we note the paper \cite{levin2024effect}, which develops a general framework for the relationship between points satisfying first- or second-order necessary conditions for a reformulation (e.g. \eqref{eq:bc.dss}) and the points satisfying first- or second-order necessary conditions for the original formulation (e.g. \eqref{eq:bc}) for a variety of problems.
Specific to the relationship of points satisfying necessary optimality conditions for \eqref{eq:bc} and \eqref{eq:bc.dss}, they show in \cite[Proposition 2.7]{levin2024effect} that  a 2NP $F$ of \eqref{eq:bc.dss} gives a 1NP $FF^\top$ of \eqref{eq:bc}, which is almost implied by existing work described in our text above on the  Burer-Monteiro approach. 
There is a sufficient condition \cite[Proposition 2.2]{levin2024effect} for a 2NP of the reformulation to yield a 2NP for the original formulation in the general case, which essentially requires equivalence of 1Ps and is not satisfied in our case. 
With this general framework in mind, our result is potentially more surprising: No extra condition is needed for a 2NP $F$ of \eqref{eq:bc.dss} to yield a 2NP of \eqref{eq:bc}.

%

\subsection{Notation} 
\label{sec:notation}
For any second-order continuously differentiable map $f: V\rightarrow W$ between two finite dimensional vector spaces $V,W$, we define the first order directional derivative and second order directional derivative along a direction $v\in V$ at an $x\in V$ as 
\begin{align*}
   Df_x[v] & := \frac{d}{dt}\left( t \mapsto f(x+tv)\right) \vert _{t=0} \in W, \\
   D^2f_x[v,v] & := \frac{d^2}{dt^2}\left( t \mapsto f(x+tv)\right)\vert _{t=0}\in W.
\end{align*}
If $W = \R$ and $V$ has an inner product $\inprod{\cdot}{\cdot}$, the gradient $\nabla f(x) \in V$ is the unique element such that $ Df_x[v] = \inprod{\nabla f(x)}{v}$ for all $v\in V$. 
If both spaces are equipped with  inner products, we denote by $Df^*_x$ the adjoint map of the linear map $Df_x: V\rightarrow W$ that maps $v\in V$ to $Df_x[v]$. 
If $V= \R^n$, we equip it with the dot product, and if $V= \symMat^\dm$, we equip it with the trace inner product. 
The notation $e_i$ is the $i$-th coordinate vector in a proper finite dimensional space. 
The notation $I_k$ is the identity matrix of size $k\times k$ and $0_k$ denotes the vector of all zeros with length $k$. 
For any two square matrices $A$ and $B$, the symmetric product is $A\circ B := \frac{1}{2}(AB^\top +BA^\top)$. The (Moore-Penrose) pseudoinverse of a matrix $A$ is denoted as $A^\dagger$.
Given a matrix $X \in \symMat^\dm$ with rank $r$, we use $V_X \in \RR^{d\times (d-r)}$ to denote a matrix with orthonormal columns whose column space is the \emph{null space} of $X$.

\subsection{Outline} \label{sec:outline}

In Section~\ref{sec: DS}, we study the problem \eqref{eq:bc} which has a single matrix variable and a single constraint (the semidefiniteness constraint) and its squared-variable analogs \eqref{eq:bc.dss} and \eqref{eq:bc.ssv-sym} (containing nonsymmetric and symmetric matrices, respectively).
Section~\ref{sec: ssv} considers \eqref{eq:f} and its squared-variable formualtions \eqref{eq:ssv} and \eqref{eq:ssv-sym}, again involving nonsymmetric and symmetric matrices, respectively.
In all cases, we focus points satisfying second-order necessary conditions for the various formulations, and the relationships between such points for the original formulation and the corresponding points in the squared-variable formulations.
Section~\ref{sec:2.3} considers an application to problems containing nuclear norm regularization terms, which arise in interesting applications.

Appendix~\ref{app:A} proves a technical result for Section~\ref{sec: DS}, while Appendix~\ref{app:B} describes relationships between the {\em local minimizers} of the three formulations considered in Section~\ref{sec: DS}.
Appendix~\ref{app:C} considers a stronger form of second-order conditions for \eqref{eq:f} that has been used by other authors, showing that it reduces to the second-order conditions considered in the majority of this paper when a strict complementarity condition holds.
Appendix~\ref{app:D} explores further the example of Section~\ref{sec:2.3}, showing that correspondences also exist between first-order points and global solutions of the original problem and its reformulation (and not just between the second-order points that are the focus of this paper).

\section{Direct substitution} \label{sec: DS}
In this section, we describe relationships between 2NPs of \eqref{eq:bc}, \eqref{eq:bc.dss}, and \eqref{eq:bc.ssv-sym}.
After defining the second-order necessary conditions for the various formulations, we explore correspondences between second-order points of   \eqref{eq:bc} and \eqref{eq:bc.dss} in \Cref{sec:2.1}. 
In \Cref{sec:2.2}, we look at correspondences between second-order points of \eqref{eq:bc} and \eqref{eq:bc.ssv-sym}. An application to nuclear norm minimization is given in \Cref{sec:2.3}.

\subsection{Equivalence between 2NPs of \eqref{eq:bc} and \eqref{eq:bc.dss}} 
\label{sec:2.1}

First- and second-order conditions for \eqref{eq:bc} and \eqref{eq:bc.dss} are defined as follows.

\begin{definition}[First order and second order conditions of \eqref{eq:bc}]\label{def:bc}
We say that $X\in \symMat^\dm$ satisfies {\em first-order} conditions (1C) for \eqref{eq:bc} if 
\begin{equation} \label{eq: bc-foc}\tag{BC-1C}
X\succeq 0, \; \nabla h(X) \succeq  0,\; \text{and} \;\nabla h(X)X =0. 
\end{equation}
We say that $X \in \symMat^\dm$ satisfies {\em second-order necessary} conditions (2NC) for \eqref{eq:bc} if
it satisfies \eqref{eq: bc-foc} and in addition 
\begin{equation}\label{eq:bc-soc}\tag{BC-2NC}
D^2 h_X[W,W] + 2\tr(WX^\dagger W \nabla h(X)) \geq 0,
\end{equation}
for all $W \in \symMat^{\dm}$ such that  $V_X^\top  WV_X =0$. 
(See \Cref{sec:notation} for definitions of $X^\dagger$ and $V_X$.)
\end{definition}

\begin{definition}[First order and second order conditions of \eqref{eq:bc.dss}]\label{def:bcssv}
We say that $(x,F)\in \RR^n \times \RR^{d \times d}$ satisfies {\em first-order} conditions (1C) for \eqref{eq:bc.dss} if $\nabla g(F) =0 $ or equivalently,
\begin{equation} \label{eq: bc.ssv-foc}\tag{DSS-1C}
\nabla g(F) = 2 \nabla h(FF^\top )F =0.
\end{equation}
We say that $F$ satisfies {\em second-order necessary} conditions (2NC) for \eqref{eq:bc.dss} if
it satisfies \eqref{eq: bc.ssv-foc} and in addition 
$D^2 g_F[\Delta ,\Delta ] \geq 0$ for all $\Delta  \in \R^{\dm \times \dm}$ (not necessarily symmetric), where
\begin{equation}\label{eq:bc.ssv-soc}\tag{DSS-2NC}
D^2 g_F[\Delta ,\Delta ] = 
D^2 h_X[W,W] + 2\tr(\nabla h(FF^\top)\Delta  \Delta ^\top ) \geq 0,
\end{equation}
where $W = F\Delta ^\top + \Delta  F^\top$. 
\end{definition}

We now introduce two technical results, then state the formal equivalence between 2NPs as Theorem~\ref{th:bc}.

\begin{lemma}\label{lem: nonnegativeT2}
Suppose a rank $r$ matrix $X\in \symMat^\dm$ has a factorization $X = FF^\top$ for some $F\in \real^{\dm \times k}$ with $k\geq r$ and $S\in \symMat^\dm $ satisfies $SF=0$ and $S\succeq 0$. For any $\Delta \in \real^{\dm \times \dm}$, let $W = F\Delta ^\top + \Delta F^\top$. Then 
$\tr(S (\Delta \Delta^\top - WX^\dagger W))\geq 0$. 
\end{lemma}
\begin{proof}
    We first simplify the term 
$\tr(S WX^\dagger W)$ using $W =  F\Delta ^\top + \Delta  F^\top$:
\begin{equation}
    \begin{aligned}
\tr(S WX^\dagger W)  & = \tr( S WX^\dagger ( F\Delta ^\top + \Delta  F^\top) )\\ 
&       = \tr( S WX^\dagger  F\Delta ^\top) + \tr( S WX^\dagger \Delta  F^\top) \\
& \overset{(a)}{=} \tr( S WX^\dagger  F\Delta ^\top)\\
& \overset{(b)}{=} 
\tr(S \Delta  F^\top X^\dagger  F\Delta ^\top). 
    \end{aligned}
\end{equation}
where $(a)$ uses the cyclic property of trace and  $SF = 0$, and  $(b)$ uses $S  F =0$ and the definition of $W$ again. 
%
%
Hence, we have
\begin{equation*}
  \tr(S(\Delta \Delta^\top - W X^\dagger W)) 
  = \tr(S \Delta  (I - F^\top X^\dagger F)\Delta  ) 
      = \inprod{S}{ \Delta  (I - F^\top X^\dagger F)\Delta^\top }.
\end{equation*}
Since $(I - F^\top X^\dagger F)\succeq 0$ by the properties of the pseudoinverse $X^\dagger = (FF^\top)^\dagger$, we know that
$\Delta  (I - F^\top X^\dagger F)\Delta^\top \succeq 0$. Because 
$S \succeq 0$ by our assumption, we have the desired result.
\end{proof}

Our second technical result has a more complicated proof, which is the subject of Appendix~\ref{app:A}.
\begin{lemma}\label{lem: constructionDelta}
Suppose a matrix $X \in \symMat^\dm$ with rank $r \le d$ has a factorization $X = FF^\top$ for some $F\in \real^{\dm \times k}$ with $k\geq r$ and $S\in \symMat^\dm $ satisfies $SX=0$. Then for any $W\in \symMat^\dm$ such that $V_X^\top WV_X =0$,
there is some $\Delta \in \real^{\dm \times k}$ such that (i) $W = \Delta  F^\top + F \Delta ^\top$, and (ii) $\tr(S (WX^\dagger W -\Delta \Delta ^\top)) =0$. 
(See \Cref{sec:notation} for definitions of $X^\dagger$ and $V_X$.)
Moreover, if the matrix $F$ is symmetric and satisfies the eigenvalue condition \eqref{eq:ec} (that is, no pair of nonzero eigenvalues of $F$ sums to $0$), then $\Delta$ can be chosen to be symmetric as well. 
\end{lemma}

We are ready for the main theorem of this section, which concerns an equivalence between the second-order conditions stated in Definitions~\ref{def:bc} and \ref{def:bcssv}.
\begin{theorem}
    \label{th:bc}
    If a matrix $X \in \symMat^\dm$ satisfies 2NC for \eqref{eq:bc}, then  any $F \in \RR^{\dm \times \dm}$ with $FF^\top =X$ satisfies 2NC for \eqref{eq:bc.dss}. Conversely, 
if a matrix $F \in \RR^{\dm \times \dm}$ satisfies 2NC for \eqref{eq:bc.dss}, then the matrix $X=FF^\top$ satisfies 2NC for \eqref{eq:bc}. 
\end{theorem}

\begin{proof}
We prove the two claims in turn.

\underline{\textbf{\eqref{eq: bc-foc}, \eqref{eq:bc-soc} yields \eqref{eq: bc.ssv-foc}, \eqref{eq:bc.ssv-soc}.}} We first show the first order condition \eqref{eq: bc.ssv-foc} hold,  that is, $\nabla h(X)F=0$. 
From the complementarity condition in \eqref{eq: bc-foc}, we know 
$\nabla h(X) X = 0.$ We thus have
\begin{equation}
\begin{aligned}
\label{eq:fg1}
\nabla h(X) X = 0 
&\implies \nabla h(X) FF^\top = 0 \implies \nabla h(X) FF^\top \nabla h(X) =0 \\
&\implies (\nabla h(X) F) (\nabla h(X) F)^\top = 0 
\implies \nabla h(X) F=0,
    \end{aligned}
\end{equation}
as required.

Next, we show \eqref{eq:bc.ssv-soc} also holds for any $F$ with $X = FF^\top$. For any $\Delta  \in \RR^{\dm \times \dm}$, we let $W =  F\Delta ^\top + \Delta  F^\top$. The LHS of \eqref{eq:bc.ssv-soc} can be rearranged as 
\begin{equation}
\begin{aligned} 
    D^2 h_X[W,W] + 2\tr(\nabla h(X)\Delta  \Delta ^\top ) 
  = &  \underbrace{D^2 h_X[W,W] + 2\tr(WX^\dagger W\nabla h(X))}_{T_1}\\
   &+  \underbrace{2\tr(\nabla h(X)(\Delta  \Delta ^\top-WX^\dagger W))}_{T_2}.
\end{aligned}
\end{equation}
Here $T_1\geq 0$ due to \eqref{eq:bc-soc} as $W = F\Delta ^\top + \Delta  F^\top $ satisfies $V_X^\top WV_X =0$, since $F^\top $ and $X$ share the same null space. 

We prove that $T_2 \ge 0$ by invoking Lemma~\ref{lem: nonnegativeT2}. Note that $\nabla h(X) \succeq 0$ by \eqref{eq: bc-foc} and $\nabla h(X) F=0$ as noted above, so the conditions of the lemma are satisfied with $S = \nabla h(X)$. The result of the lemma then shows that $T_2 \ge 0$, completing this part of the proof.

\underline{\textbf{\eqref{eq: bc.ssv-foc}, \eqref{eq:bc.ssv-soc} yields  \eqref{eq: bc-foc}, \eqref{eq:bc-soc}.}} 
We  show first that \eqref{eq: bc-foc} are satisfied.
Because $X = FF^\top$, $X$ is feasible with respect to \eqref{eq:bc}. From the condition $\nabla h(X)F=0$ in \eqref{eq: bc.ssv-foc}, we see $\nabla h(X) X =0$ as well. 
It remains to show that $\nabla h(X)\succeq 0$. 
If $F$ is invertible, then from $\nabla h(X)F=0$, we know $\nabla h(X) =0$, which is PSD. 
If $F$ is not invertible (or not full rank), then there exists $v\not =0$ such that $Fv=0$. For arbitrary $w\in \RR^\dm$, we form $\Delta  = w v^\top$. 
Substituting this choice of $\Delta $ into \eqref{eq:bc.ssv-soc}, noting that  $F\Delta ^\top = Fv w^\top = 0$, and using the definition $W = F\Delta ^\top + \Delta  F^\top$ from Definition~\ref{def:bcssv}, we have 
\begin{equation}
\begin{aligned} \label{eq: dss.proof.nabla_psd}
    D^2 h_X[W,W] + 2\tr(\nabla h(X)\Delta  \Delta ^\top )  \geq 0
   \overset{(a)}{\iff} &2 \twonorm{v}^2\tr(\nabla h(X) ww^\top )\geq 0 \\
   \overset{(b)}{\iff} & \tr(\nabla h(X) ww^\top )\geq 0.
   \end{aligned} 
\end{equation}
The equivalence $(a)$ uses $W=0$ by the choice of $\Delta $. The last equivalence $(b)$ uses the fact that $v\not =0$. Note that last condition is simply the definition of $h(X)\succeq 0$, since $w \in \RR^d$ is arbitrary.
We have therefore completed the demonstration that $X=FF^\top$ satisfies the first order conditions \eqref{eq: bc-foc}.

We show now that the condition \eqref{eq:bc-soc} holds. By comparing \eqref{eq:bc-soc} and \eqref{eq:bc.ssv-soc}, the result will hold if we can show that for any symmetric $W$ such that $V_X^\top W V_X=0$, there is $\Delta  \in \RR^{\dm\times \dm}$ such that $W = \Delta  F^\top + F \Delta ^\top$ and 
\begin{equation}\label{eq: bc-soc.verification}
\begin{aligned}
 D^2 h_X[W ,W] + 2\tr(W X^\dagger W \nabla h(X)) & \geq 0 \\
  \iff \underbrace{D^2 h_X[W ,W] + 2\tr(\nabla h(X) \Delta  \Delta ^\top )}_{T_1}  + 2\underbrace{\tr(\nabla h(X)  (W X^\dagger W - \Delta \Delta ^\top))}_{T_2} & \geq 0. 
\end{aligned}
\end{equation} 
For any $\Delta$ with the required property, we have $T_1 \geq 0 $ from \eqref{eq:bc.ssv-soc}.   
Hence  it remains to find $\Delta $ such that  (i) $W = \Delta  F^\top + F \Delta ^\top$, and (ii) $T_2\geq 0$. 
Both these claims follow from Lemma~\ref{lem: constructionDelta} with $S=\nabla h(X)$, completing the proof.
(In fact for the value of $\Delta$ constructed in Lemma~\ref{lem: constructionDelta}, we have $T_2=0$.) 
\end{proof}

\begin{remark}
    In the proof above, the square shape of  $F$ is used only in the paragraph containing \eqref{eq: dss.proof.nabla_psd}, to argue that $\nabla h(X)\succeq 0$. 
    For rectangular $F$, that is, $F\in \real^{\dm \times k}$ with $k\leq \dm$, the equivalence between the 2NPs will continue to hold provided that $\rank(F)<k$.
    The same argument regarding the existence of $v \ne 0$ in the null space of $F$ would apply, with its subsequent definition of the matrix $\Delta$ from which the conclusion in \eqref{eq: dss.proof.nabla_psd} follows.
\end{remark}

\begin{exm}\label{exm: bc.ssv}
Consider the Burer-Monteiro approach which sets the factor $F$ of size $d\times k$ with $k\leq d$ as shown in \eqref{eq: bc.ssv.bm}. 
We show here that there is a convex function $h$ such that for any $k< d$, there is a 2NP $F$ of \eqref{eq: bc.ssv.bm} such $FF^\top$ is not a 1P of \eqref{eq:bc}, or equivalently, due to convexity of $h$, not a global optimal point of \eqref{eq:bc}. 
In this example, the 1P of \eqref{eq:bc} is unique and has rank $1$.
Hence, our insistence on a square-shaped $F$ in \eqref{eq:bc.dss} is necessary for the equivalence between 2NPs. 
This example is inspired by the example \cite[Section~3]{bhojanapalli2018smoothed}, which deals only with the case of $k=d-1$.

For any $d\geq 3$, define the following matrices $A_i$, $i=1,2,\dotsc,\dm+1$ and a vector $b \in \RR^{\dm+1}$:

\begin{equation*}
    \begin{aligned}
        A_i  & = e_i e_\dm^\top + e_\dm e_i^\top,\quad i = 1,\dots, \dm-1; \quad
        A_\dm  = \epsilon \begin{bmatrix}
            I_{\dm-1} & 0 \\ 
            0 & 1
        \end{bmatrix}; \quad 
        A_{\dm+1} = \epsilon \begin{bmatrix}
            2I_{\dm-1} & 0 \\ 
            0 & 1
        \end{bmatrix},\\
    b& = \epsilon \frac{5(\dm-1)}{3} \begin{bmatrix}
        0_{\dm-1} \\ 
        1\\
       1
    \end{bmatrix}.
     \end{aligned}
\end{equation*}
Here, the vectors $e_i$, $i=1,\dots, \dm$ are standard basis vectors in $\RR^\dm$ and the number $\epsilon = \frac{1}{2}\sqrt{\frac{6}{k}}>0$. 

Consider the following optimization problem: 
\begin{equation}\label{eq: fOriginalExm2}
    \min_{X\succeq 0}\; h(X):\,= \frac{1}{2} \sum_{i=1}^{\dm+1} (\inprod{A_i}{X}-b_i)^2 =\frac{1}{2} \twonorm{\mathcal{A}(X)-b}^2.
\end{equation}
The above admits a unique minimizer $X_\star = \begin{bmatrix}
    0 & 0\\
    0 & \frac{5(\dm -1)}{3}
\end{bmatrix}$. 
To verify this claim, note that $h(X_\star)=0$ and $X_\star \succeq 0$, so $X_\star$ is a global minimizer. To see uniqueness of $X_\star$, note that $\inprod{A_\dm }{X}=b_\dm $ and $\inprod{A_{\dm+1}}{X} = b_{\dm +1}$ imply that the sum of the upper $\dm-1$ diagonal of $X$ is $0$ and the last diagonal entry is $\frac{5(\dm-1)}{3}$.  Combining these facts with $X \succeq 0$ shows that $X=X_\star$ is the only matrix satisfying $h(X)=0$. 
Note that $X_\star$ is rank $1$.

Now for each $k$, we claim that $F_k = \eta_k\begin{bmatrix}
    I_k \\ 0_{(n-k)\times {k}}
\end{bmatrix}$  for some $\eta_k$ is a 2NP of \eqref{eq: bc.ssv.bm} with the function $h$ defined in \eqref{eq: fOriginalExm2}. It is not globally optimal as $g(F_k) = h(F_kF_k^\top) >0$. The first-order conditions of 
\eqref{eq: bc.ssv.bm} at $F\in \real^{\dm\times k}$ is 
\[
\nabla g(F)  = \left(\Amap^* (\Amap(FF^\top)-b)\right)F = 0.
\]
The second-order condition is that for 
for any $\Delta \in \real^{\dm \times k}$, we have
\begin{align*}
    D^2 g_F[\Delta, \Delta ] & = 2\inprod{\Delta \Delta ^\top}{\Amap^*
    (\Amap(FF^\top)-b)} + \twonorm{\Amap(F\Delta ^\top + \Delta F ^\top)}^2 \geq 0.
\end{align*}
Simple calculation shows that 
\begin{equation*}
\begin{aligned} 
    \Amap(F_kF_k^\top) - b & = (0,\dots, 0, \epsilon \left(\eta_k^2 k - \frac{5(\dm-1)}{3}\right) , \epsilon\left( 2\eta_k^2k-\frac{5(\dm-1)}{3}\right))^\top \\
\implies  \Amap^* \left(\Amap (F_kF_k^\top) - b\right) & =
\begin{bmatrix}  
\epsilon^2\left(5 k\eta_k^2 - 5(\dm-1)\right) I_{\dm-1}  & 0 \\
0 & \epsilon^2\left(3 k\eta_k^2 - \frac{10}{3}(\dm-1)\right) 
\end{bmatrix}.
\end{aligned} 
\end{equation*}
By choosing $\eta_k^2 = \frac{\dm-1}{k}$, we see that 
\begin{equation} \label{eq:js6}
    \Amap^* \left(\Amap (F_kF_k^\top) - b\right) = -\frac{\epsilon^2(\dm-1)}{3} e_\dm e_\dm^\top,
\end{equation} 
so $\nabla g(F_k) = 0$ and the first-order conditions are satisfied.

Next, we verify the second-order condition. Given any $\Delta \in \real^{\dm \times k}$, we write 
\[
\Delta = \begin{bmatrix}
        u_1^\top \\ \vdots \\ 
        u_\dm^\top
    \end{bmatrix},
\]
where each $u_i \in \real^k$, $i=1,2,\dotsc,\dm$. Then we have 
\begin{equation}
\Amap(F_k \Delta ^\top + \Delta F_k^\top ) = (2\eta_k u_\dm,0,\dots,q_1, q_2)^\top 
\end{equation}
for some $q_1,q_2 \in \RR$, so for the second term in the definition of $D^2 g_F[\Delta,\Delta]$, we have $\twonorm{\Amap(F\Delta ^\top + \Delta F ^\top)}^2 = 4 \eta_k^2 \twonorm{u_\dm}^2 +q_1^2 + q_2^2$. 
For the first term in this definition, we have from \eqref{eq:js6} that
\[
2 \inprod{\Delta \Delta ^\top}{\Amap^*    (\Amap(F_k F_k^\top)-b)} =
-\frac{2\epsilon^2(\dm-1)}{3} \inprod{\Delta \Delta ^\top}{e_d e_d^\top} =-\frac{2\epsilon^2(\dm-1)}{3}  \| u_d \|_2^2.
\]
By combining these results, and using  the value $\eta_k = \frac{d-1}{k}$ from the first-order conditions, we find that 
\begin{align*}
 D^2 g_F[\Delta, \Delta ] & = -\frac{2(d-1)}{3}\epsilon^2 \twonorm{u_\dm}^2 +4 \eta_k^2 \twonorm{u_\dm}^2 +q_1^2 + q_2^2 \\
& \geq 
 \left(4\eta_k^2 - \frac{2(\dm-1)}{3}\epsilon^2\right)\twonorm{u_\dm}^2 \\
 &=
 \left( 4 \frac{d-1}{k} - \frac{2(\dm-1)}{3}\epsilon^2\right)\twonorm{u_\dm}^2.
\end{align*}
Using the choice $\epsilon = \frac{1}{2}\sqrt{\frac{6}{k}}$, we have $D^2 g_F[\Delta,\Delta]\geq 0$, for any choice of $\Delta$.
\end{exm}

\subsection{Relationship between 2NPs of \eqref{eq:bc} and \eqref{eq:bc.ssv-sym}} 
\label{sec:2.2}

We define the second-order necessary conditions of \eqref{eq:bc.ssv-sym}. Recall  the symmetric product of two square matrices $A$ and $B$ is
$A\circ B = \frac{1}{2}\left (AB^\top + BA^\top\right)$.

\begin{definition}
    [First order and second order conditions of \eqref{eq:bc.ssv-sym}]\label{def:bcssv-sym}
We say that $(x,F)$ satisfies {\em first-order} conditions (1C) for \eqref{eq:bc.ssv-sym} if $\nabla g(F) =0 $ or equivalently
\begin{equation} \label{eq: bc.ssv-sym-foc}\tag{DSS-Sym-1C}
\nabla g(F) = 2\nabla h(F^2)\circ F =0.
\end{equation}
We say that (x,F) satisfies {\em second-order necessary} conditions (2NC) for \eqref{eq:bc.ssv-sym} if
it satisfies first-order conditions and in addition 
$D^2 g_F[\Delta ,\Delta ] \geq 0$ for any $\Delta  \in \symMat^\dm$, or equivalently,
\begin{equation}\label{eq:bc.ssv-sym-soc}\tag{DSS-Sym-2NC}
D^2 g_F[\Delta ,\Delta ] = D^2 h_X[W,W] + 2\tr(\nabla h(F^2)\Delta  ^2) \geq 0,
\end{equation}
where $W = 2 F\circ \Delta $ and $\Delta  \in \symMat^{\dm}$ is arbitrary. 
\end{definition}

Our next theorem shows the relationship between the 2NPs of \eqref{eq:bc} and \eqref{eq:bc.ssv-sym}.

\begin{theorem}
    \label{th:bc-ssv.sym}
If a matrix $X$ satisfies 2NC for \eqref{eq:bc}, then any $F \in
\symMat^{\dm}$ such that $F^2 =X$ satisfies 2NC for
\eqref{eq:bc.ssv-sym}. Conversely, if a matrix $F$ satisfies 2NC for
\eqref{eq:bc.ssv-sym}, then the matrix $X=F^2$ satisfies 1C for
\eqref{eq:bc}. If, in addition, $F$ satisfies the eigenvalue condition
\eqref{eq:ec}, then $X = F^2$ also satisfies 2NC for \eqref{eq:bc}.
\end{theorem}
\begin{proof}
We show first that if $X$ satisfies 2NC for \eqref{eq:bc}, then any
$F$ with $F^2= X$ satisfies 2NC for \eqref{eq:bc.ssv-sym}. We then
prove the reverse direction.

\smallskip

\noindent\underline{\textbf{\eqref{eq: bc-foc}, \eqref{eq:bc-soc}  yields \eqref{eq: bc.ssv-sym-foc}, \eqref{eq:bc.ssv-sym-soc}}}. 
From the complementarity condition in \eqref{eq: bc-foc}, $\nabla h(X) X = 0$, we see that the column space of $X$ lies in the null space of  the symmetric matrix $\nabla h(X)$. Since $F$ and $X$ share the same column space, we see that $\nabla h(X) F=0$, so $\nabla h(X) \circ F =0$ and \eqref{eq: bc.ssv-sym-foc} holds.

The condition  \eqref{eq:bc.ssv-sym-soc} can be proved in the same way as for the first implication in Theorem~\ref{th:bc}. 

\smallskip

\noindent\underline{\textbf{\eqref{eq: bc.ssv-sym-foc}, \eqref{eq:bc.ssv-sym-soc} yields \eqref{eq: bc-foc}}}.
We now prove that if $F$ is a 2NP of \eqref{eq:bc.ssv-sym}, then $X = F^2$ is a 1P of \eqref{def:bc}.
First, the matrix $X=F^2$ is clearly PSD, hence feasible for \eqref{eq:bc}. 
We show later that the condition $\nabla h(X)X=0$ holds.
Next, we tackle the main technical difficulty, which is to show that $\nabla h(X)\succeq 0$.

Let $r$ be the rank of $F$ and define the eigenvalue decomposition of the symmetric matrix $F$ to be $F = U  \Sigma U^T$, where
\[
\Sigma = \diag(\sigma_1,\dots,\sigma_r,\sigma_{r+1},\dots, \sigma_\dm)  =
\begin{bmatrix}
    \hat{\Sigma} & 0 \\ 0 & 0
\end{bmatrix}, \;\; \mbox{where} \;
\hat{\Sigma} = \diag(\sigma_1,\dots,\sigma_r).
\]
where $\sigma_1, \dotsc, \sigma_r$ are the nonzero eigenvalues of $F$ and $\sigma_{r+1} = \dotsc = \sigma_d = 0$. 
(Note that $U$ is orthogonal, by the properties of the eigenvalue decomposition.) 
Since $X = F^2$, we have $X = U \Sigma^2 U^\top$

Let us introduce the following transformed versions of $\nabla h(X) \circ F$ and $\nabla h(X)$:
\begin{equation}
H := U^\top(\nabla h(X) \circ F)U  \quad \text{and}\quad  
T := U^\top \nabla h(X) U,
\end{equation}
and note that $H=0$ by \eqref{eq: bc.ssv-sym-foc}.
Then we have 
\[
H = \frac12 \left( (U^\top \nabla h(X) U) (U^\top F U) +  (U^\top F U) (U^\top \nabla h(X) U) \right) = \frac12 ( T \Sigma + \Sigma T),
\]
so that 
\begin{equation}\label{eq: nabla_f_F_nablaf}
0 = H_{ij}= \frac{1}{2} (\sigma_i +\sigma_j) T_{ij}\quad \text{for $i,j=1,2,\dotsc,\dm$.}
\end{equation}
It follows from this relation that $T_{ij}$ can be nonzero under only when $\sigma_i +\sigma_j=0$, which happens only when (1) $i$ and $j$ are both in the range $r+1,\dotsc,\dm$ (that is, in the lower $(d-r)\times (d-r)$ block of $T$) or (2) $i$ and $j$ are both in the range $1,2,\dotsc,r$ and   $\sigma_i+\sigma_j=0$. Hence, to show $\nabla h(X) \succeq 0$, which is equivalent to $T \succeq 0$, it suffices to show that (i) the lower right $(\dm-r) \times (\dm-r)$ block of $T$ (that is, $T_{(r+1):d,(r+1):d}$) is positive semidefinite and (ii) $T_{ij}=0$ for those  $i,j  \in \{1,2,\dotsc,r\}$ for which $\sigma_i+\sigma_j=0$. 
Note that if the eigenvalue condition \eqref{eq:ec} is satisfied, then (ii) is automatically implied by \eqref{eq: nabla_f_F_nablaf}.

\smallskip

\noindent\textbf{Item (i)}
In the condition \eqref{eq:bc.ssv-sym-soc}, define $\Delta $ to be 
\begin{equation}
U^\top \Delta U 
=\begin{bmatrix}
    0 & 0 \\
    0 & \hat{T}
\end{bmatrix},
\end{equation}
where $\hat{T} \in \symMat^{d-r}$ is arbitrary. Note that for such $\Delta $ we have for $W$ defined in Definition~\ref{def:bcssv-sym} that 
\begin{align*}
U^\top W U & = 
\left( (U^\top F U) (U^\top \Delta U) + (U^\top \Delta U) (U^\top F U) \right) \\
& =
\begin{bmatrix}
    \hat{\Sigma} & 0 \\
    0 & 0
\end{bmatrix}
\begin{bmatrix}
    0 & 0 \\
    0 & \hat{T}
\end{bmatrix} +
\begin{bmatrix}
    0 & 0 \\
    0 & \hat{T}
\end{bmatrix}
\begin{bmatrix}
    \hat{\Sigma} & 0 \\
    0 & 0
\end{bmatrix} = 0,
\end{align*}
so that $W=0$. 
Hence, the condition \eqref{eq:bc.ssv-sym-soc} reduces to 
\[
\inprod{T_{(r+1):d,(r+1):d}}{\hat{T}^2}\geq 0.
\]
Since $\hat{T}$ is arbitrary, we conclude that $T_{(r+1):d,(r+1):d}$ is positive semidefinite, as required.

\smallskip 

\noindent\textbf{Item (ii)} For each pair of indices $i,j \in \{1,2,\dotsc,r\}$ with   $\sigma_i+\sigma_j=0$, we define the matrix $\Delta$ in \eqref{eq:bc.ssv-sym-soc} in such a way that  $[U^\top \Delta U]_{ii} =[U^\top \Delta U]_{jj}  = \epsilon$ (for any $\epsilon$, positive or negative) and $[U^\top \Delta U] _{ij}  = [U^\top \Delta U] _{ji} =1$, with all other elements of $U^\top \Delta U$ being zero.
A simple calculation (using again that $\sigma_i+\sigma_j=0$) shows again that for the matrix $W$ of Definition~\ref{def:bcssv-sym}, we have $U^T W U =  2 U^\top (F\circ \Delta)  U$ has only two nonzeros: $2\epsilon \sigma_i$ at the $(i,i)$ location and $2 \epsilon \sigma_j$ at the $(j,j)$ location. 
We define $\bar{W} = \epsilon^{-1} W$ and note that $\bar{W}$ is independent of $\epsilon$.
We also have 
\[
\inprod{\nabla h(X)}{\Delta ^2} = \tr \left( (U^T \nabla h(X) U)^\top (U^T \Delta U)^2  \right) = \tr (TM),
\]
where $M$ is a matrix with $M_{ii}=M_{jj} = 2+\epsilon^2$ and $M_{ij}=M_{ji}=2 \epsilon$, with all other elements being zero. 
Using $T_{ii}=T_{jj}=0$ for $i,j \in \{1,2,\dotsc,r\}$, we have
$\inprod{\nabla h(X)}{\Delta ^2} = \tr (TM) = 4 \epsilon T_{ij}$.
From \eqref{eq:bc.ssv-sym-soc}, we now have
\[
0 \le D^2 h_X [W,W] + 2 \tr(\nabla h(X) \Delta^2) = \epsilon^2 D^2 h_X [\bar{W},\bar{W}]  + 8 \epsilon T_{ij}.
\]
By arbitrariness of $\epsilon$, this condition implies that $T_{ij}=0$.

For the remaining first order condition, we need to show that $\nabla h(X) X =0$.
From the structure of $T = U^\top \nabla h(X) U$ proved above, and from $X = U^\top \Sigma^2 U$, we have
\[
U^\top \nabla h(X) X U = (U^\top \nabla h(X) U) (U^\top X U) =
\begin{bmatrix}
    0 & 0 \\ 0 & T_{(r+1):d,(r+1):d} 
\end{bmatrix}
\begin{bmatrix}
    \hat{\Sigma}^2 & 0 \\ 0 & 0 
\end{bmatrix} =0,
\]
proving that $\nabla h(X) X =0$, thus completing the proof that \eqref{eq: bc-foc} are satisfied.

(We summarize the  argument above showing $\nabla h(X)\succeq 0$ as Lemma~\ref{lem: sym_f_psd} below, replacing $\nabla h(X)$ by $S$. 
We will need this argument again in Section~\ref{sec: ssv}.)

\smallskip

\noindent\underline{\textbf{\eqref{eq: bc.ssv-sym-foc}, \eqref{eq:bc.ssv-sym-soc} yields \eqref{eq: bc-foc}, \eqref{eq:bc-soc} under the}} \\ \underline{\textbf{eigenvalue condition \eqref{eq:ec}}}. 
To show that $X = F^2$ is a 2NP of \eqref{eq:bc} under the condition that no pair of nonzero eigenvalues of $F$ sums to zero, we can follow the steps in the proof of Theorem~\ref{th:bc} to show that  \eqref{eq:bc-soc} holds. 
The only difference is that we also need the $\Delta$ to be symmetric, but this requirement is taken care of by Lemma~\ref{lem: constructionDelta}. 
\end{proof}

\begin{lemma}\label{lem: sym_f_psd}
    Let $F\in \symMat^\dm$ and $S\in \symMat^{\dm} $ satisfy that $S\circ F =0$. 
 We have $S\succeq 0$ if one of the following conditions holds,
    \begin{itemize}
        \item for any $\Delta \in \symMat^\dm$ such that  $0=2F\circ \Delta$, we have $2\tr(S\Delta^2) \geq 0 $ and the eigenvalue condition \eqref{eq:ec} holds for $F$;
\item  there is a quadratic form $\mathcal{Q}:\symMat^\dm \times \symMat^{\dm}\rightarrow\real$ such that 
        $\mathcal{Q}[W,W] + 2\tr(S\Delta^2) \geq 0 $ for any $W$ of the form $W=2F\circ \Delta$ where $\Delta \in \symMat^\dm$ is arbitrary.
    \end{itemize}
\end{lemma}

The following example shows that even if $F$ is symmetric and a 2NP of \eqref{eq:ssv-sym}, $X = F^2$ is not necessarily a 2NP of \eqref{eq:bc} without the eigenvalue condition \eqref{eq:ec}.

\begin{exm}
    \label{exm: bc.ssv-sym}

Let $d=2$, and define $h(X) := \inprod{A}{X^2} + \inprod{B}{X}$. We will show that for some symmetric matrices $A$ and $B$, $X =I$ satisfies \eqref{eq: bc-foc} but not  \eqref{eq:bc-soc}. However, there is some symmetric matrix $F$ with $F^2 = I$ that satisfies \eqref{eq: bc.ssv-sym-foc} and \eqref{eq:bc.ssv-sym-soc} of
\eqref{eq:bc.ssv-sym} for $g(F)= h(F^2)$. Hence,  
Thus, a matrix satisfying 2N conditions for  \eqref{eq:bc.ssv-sym} does not necessarily yield a matrix satisfying 2N conditions for  \eqref{eq:bc}.

By writing down the first and second-order conditions \eqref{eq: bc-foc} and  \eqref{eq:bc-soc} for this $h$, we see that $X=I$ is a 1P but not a 2NP if $\nabla h(I) =2A + B =0$ and $A \not\succeq 0$. 

Then we claim that $F = \begin{bmatrix} 1 & 0 \\ 0 & -1 \end{bmatrix}$ is a 2NP for \eqref{eq:bc.ssv-sym} for some $A$ and $B$ satisfying $2A + B =0$ and $A \not\succeq 0$. 

The matrix $F$ defined as above is indeed a 1P of \eqref{eq:bc.ssv-sym} if $2A + B =0$, since the first order condition \eqref{eq: bc.ssv-sym-foc} reads as $\nabla h(I) \circ F =0$ and $\nabla h(I) = 2A +B =0$. 
For second-order conditions, \eqref{eq:bc.ssv-sym-soc} becomes
\begin{equation}\label{eq: example-soc}
\inprod{A}{(F\circ \Delta )(F\circ \Delta )} \geq 0.
\end{equation}
Let $\Delta  = \begin{bmatrix}
    y_1 & y_3 \\ 
    y_3 & y_2
\end{bmatrix}$, we have $F\circ \Delta  = \begin{bmatrix}
    y_1 & 0 \\ 0 & -y_2
\end{bmatrix}$ and $(F\circ \Delta )(F\circ \Delta ) = \begin{bmatrix}
    y_1^2 & 0 \\ 0 & y_2^2
\end{bmatrix}
$. Denoting $A = \begin{bmatrix}
    a_1 & a_3 \\ a_3 & a_2
\end{bmatrix}$, the condition \eqref{eq: example-soc} is 
equivalent to $a_1 y_1^2 + a_2 y_2^2 \ge 0$, for any $y_1$ and $y_2$.
Thus we have $a_1,a_2\succeq 0$, but we have no control over $a_3$. 
If we take
$A=\begin{bmatrix}
    \frac{1}{2} & -1 \\ 
    -1 & \frac{1}{2}
\end{bmatrix}$ and $B = \begin{bmatrix}
    -1 & 2 \\
    2 & -1
\end{bmatrix}$ then for \eqref{eq:bc} with the corresponding $h$, $X=I$ is a 1P but not a 2NP, while for \eqref{eq:bc.ssv-sym} with $g(F)=h(F^2)$,  $F = \begin{bmatrix} 1 & 0 \\ 0 & -1 \end{bmatrix}$ is a 2NP.
\end{exm} 

\subsection{An Application: Overparametrization in Nuclear Norm Minimization}
\label{sec:2.3}

The  nuclear norm (sum of singular values) is often used as a regularizer in matrix optimization problems to encourage low rank, by adding a multiple of it to the original smooth objective function $h:\RR^{\dm_1 \times \dm _2} \rightarrow \RR$. 
Consider
\begin{equation}\label{eq: nnm} \tag{NNM}
    \min_{X \in \RR^{\dm_1 \times \dm _2} } \, h(X) + \lambda \nucnorm{X}.
\end{equation}
where $\lambda>0$ is a tuning parameter. 
A well-known property of the nuclear norm is that for any $X\in \mathbb{R}^{\dm_1\times \dm_2}$, we have
\begin{equation}\label{eq: nucnorm_SDP}
    \nucnorm{X} = \min_{W_1 \in \symMat^{d_1},W_2\in \symMat^{d_2}} \, \frac{1}{2}\left(\tr(W_1)+\tr(W_2)\right) \text{ s.t. }\begin{bmatrix}
        W_1 & X \\ 
        X^\top & W_2
    \end{bmatrix}\succeq 0;
\end{equation}
see \cite[Lemma 3]{ding2021simplicity}. 
Using \eqref{eq: nucnorm_SDP}, we know that \eqref{eq: nnm} is equivalent (in terms of global minimizers)\footnote{Precisely, the equivalence is that (i) if $X_\star$ is a global minimizer of \eqref{eq: nnm}, then \eqref{eq: PSDreg} admits a global minimizer $\bar{X}_\star$ whose off-diagonal block is $X_\star$, and (ii) if $\bar{X}_\star$ is a global minimizer of \eqref{eq: PSDreg}, then its off-diagonal block is a global minimizer of \eqref{eq: nnm}. 
\label{ft: nnm_correspondence}} to the following formulation:
\begin{equation}\label{eq: PSDreg} \tag{NNM-BC}
\min_{\bar{X}} \, \bar{h}(\bar{X}) := h(X) + \tfrac{\lambda}{2} \tr(\bar{X})\;\; 
    \mbox{s.t.} \;\; \bar{X} = \begin{bmatrix}
         W_1 & X \\
         X^\top & W_2
     \end{bmatrix}\succeq 0,
\end{equation}
which has the form of \eqref{eq:bc}. 
In fact, we show in \Cref{app:D} (Lemma \ref{lem: eq1Pnnm}) that there is a direct correspondence between {\em first-order points} for \eqref{eq: nnm} and \eqref{eq: PSDreg} as well.

Consider now the direct substitution $\bar{X} = FF^\top$ for 
\[
F :=  \begin{bmatrix}
    Y \\ 
    Z
\end{bmatrix} \in \RR^{d\times d}, \;\; \mbox{\rm where $\dm = \dm_1 + \dm_2$, $Y \in \RR^{\dm_1 \times d}$, and $Z \in \RR^{\dm_2 \times d}$.}
\]
Then we can reformulate \eqref{eq: PSDreg}  in the manner of \eqref{eq:bc.dss} as
\begin{equation}\label{eq: NNMregDS} \tag{NNM-DSS}
     \min_{Y\in \RR^{\dm_1\times \dm}, \; Z\in \RR^{\dm_2 \times \dm}} \,  h(YZ^\top ) + \tfrac{\lambda}{2}\big(\fronorm{Y}^2+\fronorm{Z}^2\big).
\end{equation}
Theorem~\ref{th:bc} can be applied to show a direct correspondence between 2NP for \eqref{eq: PSDreg} and \eqref{eq: NNMregDS}. 
When $h$ is convex, a 2NP (indeed, even a 1P) of \eqref{eq: PSDreg} is a global optimum.
Thus, for convex $h$, our result shows that we can find the global optimum of  \eqref{eq: PSDreg} (as well as the global optimum of \eqref{eq: nnm}) by finding a 2NP of \eqref{eq: NNMregDS}. 
For {\em general} smooth $h$, from 
Theorem~\ref{th:bc} and Lemma \ref{lem: eq1Pnnm}, we can find a 1P of \eqref{eq: nnm} 
by finding a 2NP of \eqref{eq: NNMregDS}.
This claim follows from  (i) 1NPs of  \eqref{eq: PSDreg} and \eqref{eq: nnm} are equivalent, as mentioned above; and (ii) 2NPs of  \eqref{eq: PSDreg}  and \eqref{eq: NNMregDS} are equivalent, due to Theorem \ref{th:bc}.

In a statistical signal recovery setting \cite{chen2018harnessing,chi2019nonconvex}, where one aims to recover a signal matrix $X^\natural$ from a number of noisy observations, there are abundant results showing that  for \eqref{eq: nnm} with a suitable choice of $h$ (usually convex) and a proper choice of $\lambda$, the solution of \eqref{eq: nnm} is close to $X^\natural$ and is optimal  in a statistical sense; see \cite{candes2011tight,chen2020noisy,wainwright2019high}. 
{In the setting with convex $h$, We can solve \eqref{eq: nnm} by finding a 2NP of \eqref{eq: NNMregDS}.
If instead we were to use the Burer-Monteiro formulation \eqref{eq: bc.ssv.bm} to solve \eqref{eq: PSDreg} --- a more compact formulation than \eqref{eq: NNMregDS} when $k<d$ --- we would need more}
stringent conditions on $h$ (e.g., the so-called restricted isometry property) or $k$ (e.g. $k$ needs to be exactly the rank of $X^\natural$), as well as a low noise level, and possibly more observations than statistically necessary (see for example \cite{zhu2018global,ge2017no,zhang2022improved}). 
By contrast, the ``overparametrized" formulation \eqref{eq: NNMregDS} requires only convexity of $h$ and a proper choice of $\lambda$.
Its requirements on the noise level and the number of observations are no more stringent than those more classical convex approaches \cite{candes2011tight,chen2020noisy,wainwright2019high}.
The contrast between these two formulations shows the possible advantages of overparametrization in this context.

%

If we use the symmetric parametrization $\bar{X} = F^2$ in \eqref{eq: PSDreg}, where
\[
F = \begin{bmatrix}
    Y_1 &  Y_3 \\ Y_3^\top  & Y_2
\end{bmatrix}, \;\; \mbox{that is,} \; Y = \begin{bmatrix}
    Y_1 &  Y_3 \end{bmatrix} \;\; \mbox{and} \;\; Z = \begin{bmatrix} Y_3^\top  & Y_2 \end{bmatrix}, 
\]
with $Y_1$ and $Y_2$ symmetric,  then by substituting into \eqref{eq: NNMregDS}, we find that \eqref{eq: PSDreg} is equivalent (in terms of global minimizers) to the following problem: 
\begin{equation}\label{eq: NNMregDS-sym} \tag{NNM-DS-Sym}
\begin{aligned}
    \min & \; h(Y_1 Y_3 + Y_3 Y_2) + \tfrac{\lambda}{2}\big(\fronorm{Y_1}^2+\fronorm{Y_2}^2+2 \fronorm{Y_3}^2\big) \\
{\mbox{s.t.}} &\quad {Y_1\in \symMat^{\dm_1},  \; Y_2\in \symMat^{\dm_2},\; Y_3 \in \RR^{\dm_1 \times  \dm_2}}
\end{aligned}
\end{equation}
From Theorem \ref{th:bc-ssv.sym}, we know that if $F$ a 2NP of \eqref{eq: NNMregDS-sym}, then $\bar{X}= F^2$ is a 1P of \eqref{eq: PSDreg}. Hence, if $h$ is convex, then such an $\bar{X}$ is globally optimal. 

\section{Slack-Variable Formulations \eqref{eq:f}, \eqref{eq:ssv}, and \eqref{eq:ssv-sym}} 
\label{sec: ssv}
In this section, we show that versions of our main theorems, Theorems~\ref{th:bc} and \ref{th:bc-ssv.sym}, continue to hold for the more complicated problems \eqref{eq:f}, \eqref{eq:ssv}, and \eqref{eq:ssv-sym}. 
We define the following Lagrangians for \eqref{eq:f} and \eqref{eq:ssv}, respectively, making use of  a Lagrange multiplier $\Lambda \in \symMat^\dm$:
\begin{subequations}\label{eq: Lgs}
\begin{align}
    \cL(x,\Lambda) & := f(x) - \tr( \Lambda C(x)),\\
    \label{eq: Lgs.ssv}
    \cLssv (x,F,\Lambda) & := f(x) - \tr\left( \Lambda \left(C(x)-FF^\top\right)\right).
\end{align}
\end{subequations}
The quadratic form $D^2 \cL(x,\Lambda)[z,z]$ for any $x,z\in \R^n$ and $\Lambda \in \symMat^\dm$, can be written explicitly as 
\begin{equation}
\label{eq: Lg_original_quadratic_form}
  D^2 \cL(x)[z,z] 
   =  D^2 f(x)[z,z] -\tr(\Lambda (D^2 C_x[z,z])).
\end{equation}
The quadratic form $D^2 \cLssv (x,F,\Lambda)[(z,\Delta ),(z,\Delta )]$ for any $z\in \R^n$, $F,\Delta \in \R^{\dm \times \dm}$, and $\Lambda \in \symMat^\dm$ is 
\begin{multline}\label{eq: Lg_sq_slack_quadratic_form}
 D^2 \cLssv (x,F,\Lambda)[(z,\Delta ),(z,\Delta )] \\
=  D^2 f(x)[z,z]+ 2\tr(\Lambda \Delta  \Delta ^\top) -\tr(\Lambda (D^2 C_x[z,z])).
\end{multline}

\subsection{Equivalence between 2NP of \eqref{eq:f} and \eqref{eq:ssv}}

We start with definitions of the first-order and second-order conditions for the problems \eqref{eq:f} and \eqref{eq:ssv}. Note that (as defined in Section~\ref{sec:notation}) $DC_x^*(\Lambda)$ is the adjoint of the linear map $DC_x: \RR^n \to \symMat^d$.
\begin{definition}[First order and second order conditions of \eqref{eq:f}]\label{def:f}
We say that a point $(x,\Lambda) \in \real^n \times \symMat^\dm $ satisfies {\em first-order} conditions (1C) to be a solution of \eqref{eq:f} if 
\begin{equation} \label{eq: f-foc}\tag{NSDP-1C}
\nabla f(x) = 
D C_x^*(\Lambda),\; \Lambda C(x) =0,\; C(x) \succeq 0,\text{and}\; \Lambda \succeq 0.
\end{equation}
Let $r = \rank(C(x))$ and define $V \in \R^{\dm \times (d-r)}$ to be the matrix with 
orthonormal columns whose column space is the null space of $C(x)$.  
We say that $(x,\Lambda)$ satisfies {\em second-order necessary} conditions (2NC) to be a solution of \eqref{eq:f} if
it satisfies \eqref{eq: f-foc} and in addition, 
for all $z\in \RR^n$ s.t. 
$V^\top  D C_x[z]V=0$, there holds the inequality 
\begin{equation}\label{eq: f-soc}\tag{NSDP-2NC}
D^2 \cL(x)[z,z] + 2 \tr \left( DC_x[z] C(x)^\dagger DC_x[z] \Lambda \right)\geq 0.
\end{equation}
Here $C(x)^\dagger$ is the pseudo-inverse of the matrix $C(x)$.
\end{definition}

\begin{definition}[First order and second order conditions of \eqref{eq:ssv}]\label{def:ssv}
We say that $(x,F,\Lambda )$ satisfies {\em first-order} conditions (1C) to be a solution of \eqref{eq:ssv} if 
\begin{equation} \label{eq: ssv-foc}\tag{SSV-1C}
\nabla f(x) = 
D C_x^*(\Lambda),\; \Lambda F =0,\; C(x)  = FF^\top.
\end{equation}
We say that $(x,F,\Lambda)$ satisfies {\em second-order necessary} conditions (2NC) to be a solution of \eqref{eq:ssv} if
it satisfies first-order conditions and in addition, 
for all $(z,\Delta ) \in \R^n\times \R^{\dm \times \dm}$ s.t. 
$DC_x[z] = F\Delta ^\top + \Delta  F^\top$, there holds the inequality,
\begin{equation}\label{eq: ssv-soc}\tag{SSV-2NC}
D^2 \cLssv (x,F,\Lambda)[(z,\Delta ),(z,\Delta )] 
\geq 0.
\end{equation}
\end{definition}

We now state the equivalence result.
\begin{theorem} 
    \label{th:f}
    If a point $(x,\Lambda)\in \R^n \times \symMat^\dm$ satisfies 2NC for \eqref{eq:f}, then any $(x,F,\Lambda)$ with $C(x) = FF^\top$ and $F\in \R^{\dm \times \dm}$ satisfies 2NC for \eqref{eq:ssv}. Conversely, 
if $(x,F,\Lambda) \in \R^n\times \R^{\dm \times \dm}\times  \symMat^\dm$ satisfies 2NC for \eqref{eq:bc.dss}, then $(x,\Lambda)$ satisfies 2NC for \eqref{eq:f}. 
\end{theorem}
    \begin{proof}
\underline{\textbf{\eqref{eq: f-foc}, \eqref{eq: f-soc} yields \eqref{eq: ssv-foc}, \eqref{eq: ssv-soc}}}
    Suppose $(x,\Lambda)$ satisfies \eqref{eq: f-foc}, \eqref{eq: f-soc}.  Taking any $F$ such that $C(X) = FF^\top$, we show that $(x,F,\Lambda)$  satisfies \eqref{eq: ssv-foc}, \eqref{eq: ssv-soc}. 
    We show that $\Lambda C(x)=0 \implies \Lambda F=0$ by an argument like that of \eqref{eq:fg1}, namely:
    \[
\Lambda C(x) = 0 \implies \Lambda FF^\top=0 \implies \Lambda FF^\top \Lambda = 0 \implies (\Lambda F) (\Lambda F)^\top = 0 \implies \Lambda F=0.
    \]
    The remaining first-order conditions in \eqref{eq: ssv-foc} are trivially implied by \eqref{eq: f-foc}.
    To verify the condition \eqref{eq: ssv-soc}, we take any $(z,\Delta )\in \R^n \times \R^{\dm \times \dm}$ s.t. 
    $DC_x[z] = F\Delta ^\top + \Delta  F^\top$. Note that $V^\top DC_x[z] V = 0$ since $V^\top F =0$ as $F$ shares the same column space as $C(x)$. Furthermore,  
    from \eqref{eq: Lg_original_quadratic_form} and \eqref{eq: Lg_sq_slack_quadratic_form}, we see that 
    \begin{multline*}
     D^2\cLssv (x,F,\Lambda)[(z,\Delta ),(z,\Delta )]
        = \underbrace{D^2 \cL(x,\Lambda)[z,z] + 
     2\tr(DC_x[z] C^\dagger(x) DC _x(z) \Lambda)}_{T_1} \\ 
     + 2\underbrace{\tr\left((\Delta \Delta ^\top - DC_x[z] C^\dagger(x) DC _x(z)) \Lambda\right)} _{T_2}.
        \end{multline*}
Here $T_1\geq 0$ due to \eqref{eq: f-soc}. The inequality $T_2 \geq 0$ is immediate from Lemma \ref{lem: nonnegativeT2}. 

\underline{\eqref{eq: ssv-foc}, \eqref{eq: ssv-soc} yields \textbf{\eqref{eq: f-foc}, \eqref{eq: f-soc}}}
    Given  $(x,F,\Lambda)$ satisfying  \eqref{eq: ssv-foc}, \eqref{eq: ssv-soc}, we show that $(x,\Lambda)$ satisfies \eqref{eq: f-foc}, \eqref{eq: f-soc}. 
    From \eqref{eq: ssv-foc}, we see that the only missing condition in \eqref{eq: f-foc} is $\Lambda \succeq 0$. 
    When $F$ is invertible, then $\Lambda F =0 \implies \Lambda =0$, so that $\Lambda \succeq 0$ in this special  case. 
    Otherwise, there exists a  vector $v\in \R^n$ with $\|v\|_2=1$ such that $Fv =0$. 
    By setting $z=0$ and $\Delta  = wv^\top$, where $w\in \R^n$ is arbitrary, we see that the condition $DC_x[z] = 0 = F\Delta ^\top + \Delta  F ^\top$ holds for any $w$. Using \eqref{eq: ssv-soc}, we see that 
    $D^2 \cLssv (x,F,\Lambda)[(z,\Delta ),(z,\Delta )] = 2\tr(ww^\top  \Lambda )\geq 0$.
    Since $w\in \R^n$ is arbitrary, this condition implies that $\Lambda \succeq 0$. 
    
    To verify that \eqref{eq: f-soc}, for all $z\in \R^n$ such that $V ^\top DC_x[z] V =0$, we need to find  some $\Delta  \in \R^{\dm \times \dm}$ such that 
    $DC_x [z] = \Delta  F^\top + F \Delta ^\top$ and 
    \begin{multline*}
        D^2 \cL(x) [z,z] + 2\tr( DC_x[z] C^\dagger (x) DC_x[z]\Lambda ) \geq 0 
   \iff 
   \underbrace{D^2 \cLssv (x,F,\Lambda)[(z,\Delta ),(z,\Delta )]}_{T_1} \\
   + 2\underbrace{\tr\left((-\Delta \Delta ^\top  + DC_x[z] C^\dagger(x) DC _x(z)) \Lambda\right)} _{T_2} \geq 0.
    \end{multline*}
Note that $T_1\geq 0$ due to \eqref{eq: ssv-soc} so long as $\Delta $ is found. Our task now is to find $\Delta $ such that (i) $DC_x[z] = \Delta  F^\top + F\Delta ^\top$, and (ii) $T_2\geq 0$. 
But Lemma~\ref{lem: constructionDelta} yields this result immediately, completing our proof.
\end{proof}

\subsection{Relationship between 2NP of \eqref{eq:f} and \eqref{eq:ssv-sym}}
We note that the Lagrangian of \eqref{eq:ssv-sym} is simply $\cLssv $ from \eqref{eq: Lgs.ssv} except that we  restrict the variable $F$ to be symmetric. 

\begin{definition}[First order and second order conditions of \eqref{eq:ssv-sym}]\label{def:ssv-sym}
We say that $(x,F,\Lambda )\in \RR^n \times \symMat^d \times \symMat^d$ satisfies {\em first-order} conditions (1C) to be a solution of \eqref{eq:ssv-sym} if 
\begin{equation} \label{eq: ssv-sym-foc}\tag{SSV-Sym-1C}
\nabla f(x) = 
D C_x^*(\Lambda),\; \Lambda \circ F =0,\; C(x)  = F^2.
\end{equation}
We say that $(x,F,\Lambda)$ satisfies {\em second-order necessary} conditions (2NC) to be a solution of \eqref{eq:ssv-sym} if it satisfies \eqref{eq: ssv-sym-foc} and in addition,  for all $(z,\Delta) \in \R^n\times \symMat^{\dm}$ such that $DC_x[z] = 2\Delta  \circ F $, there holds the inequality,
\begin{equation}\label{eq: ssv-sym-soc}\tag{SSV-Sym-2NC}
D^2 \cLssv (x,F,\Lambda)[(z,\Delta ),(z,\Delta )] 
\geq 0.
\end{equation}
\end{definition}

Our next theorem shows the relationship between the 2NPs of \eqref{eq:f} and \eqref{eq:ssv-sym}.

\begin{theorem}
    \label{th:ssv.sym}
  If $(x,\Lambda)\in \R^n \times \symMat^\dm$ satisfies 2NC for  \eqref{eq:f}, then any $(x,F,\Lambda)\in \RR^n \times \symMat^d \times \symMat^d$ with $C(x) = F^2$ satisfies 2NC for \eqref{eq:ssv-sym}. Conversely, 
if $(x,F,\Lambda) \in \R^n \times \symMat^{\dm}\times  \symMat^\dm$ satisfies 2NC for \eqref{eq:ssv-sym}, and, in addition, $F$ satisfies the eigenvalue condition \eqref{eq:ec}, then $(x,F,\Lambda)$ also satisfies 2NC for \eqref{eq:f}.
\end{theorem}
\begin{proof}
\underline{\textbf{\eqref{eq: f-foc}, \eqref{eq: f-soc} yields \eqref{eq: ssv-sym-foc}, \eqref{eq: ssv-sym-soc}}}
Suppose  that $(x,\Lambda)$ is given and satisfies 2NC for \eqref{eq:f}. 
We shall prove that $(x,F,\Lambda)$ for any $F \in \symMat^d$ such that  $C(x) = F^2$ is an 2NP of \eqref{eq:ssv-sym}. 
The only nontrivial requirement in \eqref{eq: ssv-sym-foc} is to show that  $\Lambda \circ F =0$. 
We can argue as in the proof of Theorem~\ref{th:f} that $\Lambda C(x)=0 \implies \Lambda F=0$, and by symmetry of $\Lambda$ and $F$, we have $0 = (\Lambda F)^\top = F \Lambda$, and hence $\Lambda \circ F=0$.
To prove \eqref{eq: ssv-sym-soc}, we follow the corresponding inference in the first part of the proof of Theorem~\ref{th:f}.

\underline{\textbf{\eqref{eq: ssv-sym-foc}, \eqref{eq: ssv-sym-soc} yields \eqref{eq: f-foc}, \eqref{eq: f-soc} under the}} \\ \underline{\textbf{eigenvalue condition \eqref{eq:ec}}} 
Now suppose that $(x,F,\Lambda)$ satisfies 2NC for \eqref{eq:ssv-sym} and additionally that $F$ satisfies condition \eqref{eq:ec}. 
To show \eqref{eq: f-foc}, the only nontrivial step is to show that $\Lambda \succeq 0$. But this follows from Lemma \ref{lem: sym_f_psd}, by setting $z=0$ in \eqref{eq: ssv-sym-soc}. 
To establish  \eqref{eq: f-soc}, we can follow the same procedure as in the proof of the second part of Theorem~\ref{th:f}. 
\end{proof}

Finally, we show that the eigenvalue condition \eqref{eq:ec} $\sigma_i +\sigma_j \not =0$ is actually critical for the ``converse" direction of Theorem~\ref{th:ssv.sym}. 
Otherwise, it is possible that $(x,F,\Lambda) \in \RR^n \times \symMat^d \times \symMat^d$ is a 2NP of \eqref{eq:ssv-sym}, yet there is no $\Psi \succeq 0$ such that $(x,\Psi)$ is a 1P of \eqref{eq:f}.

\begin{exm}
\label{exm:ssv-sym}
    Consider the following problem, for which $n=1$ and $d=2$:
    \begin{equation}
        \min_{x\in \R} \;f(x):= \frac{1}{2}(x+1)^2\quad \text{s.t.}\quad  \; C(x) = \begin{bmatrix}
            x^2 + 1 & x \\ 
            x & x^2 +1 
        \end{bmatrix}\succeq 0.
    \end{equation}
    At $x=0$, we have 
\begin{equation}
    \nabla f(0) = 1,\; C(0) = I,\; DC_0[z] = z\begin{bmatrix}
            0 & 1 \\ 
            1 & 0 
        \end{bmatrix},
\end{equation}
and 
\begin{equation}
D^2 f_0[z,z] = z^2, \quad  \text{and}\quad 
        D^2 C_0[z,z] = z^2 \begin{bmatrix}
            2 & 0 \\ 
            0 & 2 
        \end{bmatrix}.
\end{equation}      
There is no $\Psi \succeq 0$ such that $(x,\Psi)$ satisfy \eqref{eq: f-foc}. 
If such a $\Psi \succeq 0$ were to exist, we would have from $C(0)\Psi =0$ that  $\Psi =0$. 
But then  $1= \nabla f(0) \ne DC_0^* \Psi  = 0$. 

However, we show that $(x,F,\Lambda) = 
\left( 0, \begin{bmatrix}
    1 & 0 \\ 
    0 & -1 
\end{bmatrix}, \begin{bmatrix}
    0 & \frac{1}{2} \\ 
    \frac{1}{2} & 0 
\end{bmatrix}\right)$ satisfies the 2NC for \eqref{eq:ssv-sym}. 
For  \eqref{eq: ssv-sym-foc}, we have $1= \nabla f(0) = DC_0^*(\Lambda) = \tr(DC_0[1]\Lambda )=1$, $\Lambda \circ F = 0$, and $C(0) = F^2$, as required.
For \eqref{eq: ssv-sym-soc}, we note that for any $(z,\Delta )\in \R \times \symMat^2$ such that 
$DC_0[z] = \Delta  F^\top + F \Delta ^\top$, we must have $z = 0$ and $\Delta  = y \begin{bmatrix}
    0 & 1 \\ 1 & 0
\end{bmatrix}$ for some $y\in \R$.
Hence, we see that 
\[
D^2 \cLssv (x,F,\Lambda)[(z,\Delta ),(z,\Delta )] 
=  \underbrace{D^2 f(x)[z,z]}_{=0 \text{ as $z=0$}}+ 2\underbrace{\tr(\Lambda \Delta  \Delta ^\top)}_{=0} -\underbrace{\tr(\Lambda (D^2 C_x[z,z]))}_{=0\text{ as $z=0$}},
\]
so that $D^2 \cLssv (x,F,\Lambda)[(z,\Delta ),(z,\Delta )] \geq 0$, as required.
\end{exm}

\section{Conclusion} \label{sec:conclusion}

We have examined matrix optimization problems with semidefiniteness constraints, along with their ``squared-variable" reformulations that replace the semidefinite variable $X$ with a product $FF^T$.
Our main results show close correspondences between second-order points for the original problems and their squared-variable reformulations.
These observations open the way to considering an alternative approach for solving the original problems: Apply an algorithm to the squared-variable reformulation that guarantees convergence to second-order points.
Many optimization algorithms converge to such points in practice, even when the convergence theory guarantees only convergence to first-order points. 
The practicality of such a scheme is a subject for future investigation.

\bibliographystyle{alpha}
\bibliography{reference}


\appendix

\section{Proof of Lemma~\ref{lem: constructionDelta}}
\label{app:A}


\begin{proof}

Define the SVD of $F$ as follows:
\[
F = U \begin{bmatrix} \Lambda & 0 \\ 0 & 0 \end{bmatrix} Y^T,
\]
where $U \in \RR^{d \times d}$ and $Y \in \RR^{k \times k}$ are orthogonal matrices and 
\begin{equation} \label{eq:ux1}
\Lambda = \diag(\lambda_1,\lambda_2,\dotsc,\lambda_r)
\end{equation}
is a diagonal matrix constructed from the nonzero singular values of $F$, satisfying $\lambda_1 \ge \lambda_2 \ge \dotsc \ge \lambda_r>0$. 
We partition $U$ as $U = [U_X \; V_X]$, where $U_X \in \RR^{d\times r}$ contains the singular vectors corresponding to the nonzero singular values of $F$. 
Since $X=FF^\top$, we have 
\begin{equation} \label{eq:ux2}
    X = U_X \Lambda^2 U_X^T, \quad X^\dagger = U_X \Lambda^{-2} U_X^T.
\end{equation}
The columns of $V_X$ span the null space of $X$ and $F^T$.

We first prove the claim (i), defining in the process the matrix $\Delta$ with the properties  claimed in the lemma.

Denoting by  $J_{s \times t}$ denotes an $s \times t$  matrix whose entries are all $1$, we define
\begin{equation} \label{eq:Ldef}
L := \begin{bmatrix}
\frac{1}{2}  J_{r\times r} & J_{r\times (d-r)} \\ 
J_{(d-r)\times r} & J_{(d-r)\times (d-r)} 
\end{bmatrix}.
\end{equation}
We define $\Delta$ as follows:
\begin{equation} \label{eq:defD}
\Delta  := U(L \odot (U^\top  WU))U^\top (F^\dagger)^{\top},
\end{equation}
where $\odot$ is the entrywise product. 
We can simplify $\Delta F^\top$ as follows.
\begin{align}
\nonumber
\Delta  F^\top  & =  U(L \odot (U^\top  WU))U^\top ((F^\dagger)^{\top} F^\top )\\ 
\nonumber
& \overset{(a)}{=} U(L \odot (U^\top  WU))U^\top U \begin{bmatrix}
    I_r & 0 \\ 
    0 & 0\\
\end{bmatrix} U^\top\\
\label{eq:si3a}
& = U(L \odot (U^\top  WU))\begin{bmatrix}
    I_r & 0 \\ 
    0 & 0\\
\end{bmatrix} U^\top,
\end{align}
In step $(a)$, we use the fact $F$ and $X$ share the same column space and the definition of pseudo-inverse. 
Recalling that the lower-right block of $U^\top W U$ is zero, because of  $V_X^\top W V_X =0$, we have for some matrices $A \in \symMat^r$ and $B \in \RR^{r \times \dm}$ that 
\begin{equation}
    \label{eq:ux3}
    U^\top W U = \begin{bmatrix}
    A & B \\ B^\top & 0
\end{bmatrix}, 
\end{equation}
and thus from \eqref{eq:Ldef}, 
\begin{equation}
    \label{eq:ux4}
    L\odot (U^\top WU) = \begin{bmatrix}
    \frac{1}{2} A & B \\ B^\top & 0
\end{bmatrix}.
\end{equation}
By substituting into \eqref{eq:si3a} and adding this quantity to its transpose, we have
\[
        \Delta  F^\top + F\Delta ^\top 
        = 
        U \begin{bmatrix}
        \frac{1}{2} A & 0 \\ B^\top & 0\\    
        \end{bmatrix} U^\top  + 
        U \begin{bmatrix}
        \frac{1}{2} A & B \\ 0 & 0\\    
        \end{bmatrix} U^\top 
=         U \begin{bmatrix}
        A & B \\ B^\top  & 0\\    
        \end{bmatrix} U^\top 
        = W,
\]
completing the proof of (i).

To show (ii), we have from the definition of $\Delta$ in \eqref{eq:defD} together with $X=FF^\top \implies X^\dagger = (F^\dagger)^\top (F^\dagger)$ that 
    \begin{align*} 
\Delta  \Delta  ^\top  & = 
U(L \odot (U^\top  WU))U^\top  (F^\dagger)^\top (F^\dagger) U (L \odot (U^\top  WU)) U^\top \\ 
& =U(L \odot (U^\top  WU))(U^\top  X^\dagger U) (L \odot (U^\top  WU)) U^\top.
    \end{align*} 
We now define 
\begin{equation}
    \label{eq:defR1R2}
    R_1 := U^\top W X^\dagger W U, \quad R_2 := U^\top \Delta \Delta^\top U,
\end{equation}
so that
\[
U (R_1-R_2) U^\top = W X^\dagger W - \Delta \Delta^\top.
\]
Expressing these matrices in terms of the matrices $A$, $B$ defined in \eqref{eq:ux3} and $\Lambda$ defined in \eqref{eq:ux1}, and noting \eqref{eq:ux2} and defining $D:= \Lambda^{-2}$, we have
\begin{align*}
  R_1 & =(U^\top  WU)(U^\top  X^\dagger U)  (U^\top  WU) \\
  & = 
  \begin{bmatrix}
    A & B \\ B^\top & 0
\end{bmatrix} 
\begin{bmatrix}
    D & 0 \\ 0 & 0
\end{bmatrix}
 \begin{bmatrix}
    A & B \\ B^\top & 0
\end{bmatrix} 
=
    \begin{bmatrix}
        A D A & A D B \\ 
        B^\top D A & B^\top D B
    \end{bmatrix}, \\
    R_2 &=(L\odot U^\top  WU)(U^\top  X^\dagger U)  (L\odot U^\top  WU) = 
    \begin{bmatrix}
       \frac{1}{4} A D A & \frac{1}{2} A D B \\ 
        \frac{1}{2}B^\top D A & B^\top D  B
    \end{bmatrix},
\end{align*}
so that
\begin{equation}
\label{eq:ux8}
    R_1 - R_2 
    = 
    \begin{bmatrix}
       \frac{3}{4} A D A & \frac{1}{2} A D B \\ 
        \frac{1}{2}B^\top D A & 0
    \end{bmatrix}.
\end{equation}
Now recall that the matrix $S \in \symMat^\dm$ in the statement of the theorem has $SX=XS=0$, so that $S$ lies in the null space of $X$, that is, the column space of $V_X$. Thus $U^\top S U$ has zeros everywhere except  in the lower right block (of dimension $(d-r) \times (d-r)$). 
We thus have
\[
\tr(S (WX^\dagger W -\Delta \Delta ^\top)) = \tr( S(U(R_1 - R_2)U^\top ) =  \tr((U^\top S U) (R_1-R_2) ) = 0,
\]
where the last step follows from the structure of $U^\top S U$ described above as well as the structure of  $R_1-R_2$ from \eqref{eq:ux8}.
This completes the proof of (ii).

For the final claim (concerning the case of  symmetric $F$), we write the eigenvalue decomposition of $F$ as 
\begin{equation}
    F = U \begin{bmatrix} \Sigma & 0 \\ 0 & 0 \end{bmatrix} U^\top,
\end{equation}
where $U \in \RR^{d \times d}$ is an orthogonal matrix and 
\begin{equation} \label{eq:diag_F}
\Sigma = \diag(\sigma_1,\sigma_2,\dotsc,\sigma_r)
\end{equation}
is a diagonal matrix constructed from the nonzero eigenvalues of $F$.
As before, we partition $U$ as $U = [U_X \; V_X]$, where $V_X \in \RR^{d\times (d-r)}$ spans the null space of $F$ and $X=F^2$.
We have 
\begin{equation} \label{eq:ux12}
    X = U_X \Sigma^2 U_X^T.
\end{equation}

We define $\Delta = U \bar\Delta U^T$, where
\begin{equation} \label{eq:barD}
    \bar\Delta_{ij} :=
    \begin{cases}
        [U^\top W U]_{ij}/(\sigma_i+\sigma_j), & \text{$i\leq r$ or $j\leq r$,}\\ 
        0, & \text{otherwise.}
    \end{cases}
\end{equation}
Note that because the eigenvalue condition is assumed, the denominators $\sigma_i+\sigma_j$ in this definition are all nonzero.
One can verify directly that $\Delta $ satisfies (i) $W = F \Delta +\Delta F $. 
To show that (ii) is satisfied by this choice of $\Delta$,  we define $R_1=U^\top W X^\dagger WU$ as in \eqref{eq:defR1R2}, and $R_2=U^\top \Delta^2 U = \bar{\Delta}^2$, also as in \eqref{eq:defR1R2}, but making use of the symmetry of $\Delta$ and the definition of $\bar\Delta$.

As in \eqref{eq:ux3}, we define $A$ and $B$ implicitly by
\[
U^\top W U = \begin{bmatrix}
    A & B \\ B^\top & 0
\end{bmatrix}
\]
and note from \eqref{eq:ux12} that 
\[
X^\dagger = U_X D U_X^T = U \begin{bmatrix}
    D & 0 \\ 0 & 0
\end{bmatrix} U^\top, \quad 
\mbox{where $D:= \Sigma^{-2} = \diag(1/\sigma_1^2,\dotsc, 1/\sigma_r^2)$.}
\]
Using this notation, we have 
\[
R_1 = \begin{bmatrix}
        ADA & ADB \\ 
        B^\top DA & B^\top D B
    \end{bmatrix}.
\]
Moreover, for some symmetric $A' \in \symMat^{r}$, we have for $\bar\Delta$ from \eqref{eq:barD} that
\[
\bar\Delta 
= \begin{bmatrix}
        A' & \Sigma^{-1} B \\
        (\Sigma^{-1}B)^\top & 0
\end{bmatrix}.
\]
Thus, we have that 
\[
R_2 = \bar\Delta^2 = 
\begin{bmatrix}
    \times & \times \\
    \times & B^\top DB
\end{bmatrix},
\]
so that 
\begin{equation} \label{eq:ux13}
R_1-R_2 = \begin{bmatrix}
    \times & \times \\
    \times & 0
\end{bmatrix},
\end{equation}
where the symbol ``$\times$" denotes certain submatrices whose contents are not relevant to this argument, and the zero block at the lower right of $R_1-R_2$ has size $(d-r) \times (d-r)$.
Hence 
\[
\tr(S (WX^\dagger W -\Delta \Delta ^\top)) = \tr( (U^\top SU)(R_1 - R_2) ) = 0,
\]
where the last step makes use of the structure of $R_1-R_2$ in \eqref{eq:ux13}, together with the fact, noted earlier in the proof, that $U^\top SU$ has nonzeros only in its lower right $(d-r) \times (d-r)$ block.
This completes the proof of the final claim.
\end{proof}

\section{Relationship between local minimizers of original and squared-variable formulations} 
\label{app:B}

It is clear that that \eqref{eq:bc}, \eqref{eq:bc.dss}, and \eqref{eq:bc.ssv-sym} are equivalent in terms of {\em global minimizers}, in the sense that a global minimizer for one of these formulations translates in an obvious way to a global minimizer for the others. 
Meanwhile, this paper focuses on the relationships between points satisfying {\em second-order conditions} for these various formulations. 
In this appendix,  we articulate the relationships between {\em local minimizers} and {\em strict local minimizers} for these formulations.

\subsection{Local minimizers of \eqref{eq:bc}, \eqref{eq:bc.dss}, and \eqref{eq:bc.ssv-sym}}
\label{sec:B1}

We start with a result on equivalence between (strict) local minimizers of \eqref{eq:bc} and \eqref{eq:bc.ssv-sym}, then follow in \Cref{thm:bc-ssv-local_min} a result about the relationship between local minimizers of \eqref{eq:bc} and \eqref{eq:bc.dss}.

\begin{theorem}\label{thm:bc-ssv-sym-local_min}
    Suppose $X \in \symMat^\dm$ is a (strict) local minimizer of \eqref{eq:bc}, then any $F$ with $X = F^2$ is a (strict) local minimizer of \eqref{eq:bc.ssv-sym}. \textcolor{black}{Conversely, if $F$ is a  local minimizer of \eqref{eq:bc.ssv-sym} and satisfies the eigenvalue condition \eqref{eq:ec}, then $X = F^2$ is a local minimizer of \eqref{eq:bc}. Moreover,  if $F$ is a strict local minimizer of \eqref{eq:bc.ssv-sym}, then $X = F^2$ is a strict local minimizer of \eqref{eq:bc}.}
\end{theorem}
\begin{proof}
Define the semidefinite cone $\symMat^\dm_+ := \{ X' \in \symMat^\dm \, : \, X' \succeq 0 \}$ and the ball $B(X,\epsilon)$ for $X \in \symMat^d$ and  $\epsilon>0$ as follows:
\begin{equation} \label{eq:sb2}
    B(X,\epsilon) := \{ X' \in \symMat^d \, : \, \|X'-X\|_F \le \epsilon \}.
\end{equation}

Suppose first that $X$ is a (strict) local minimizer of \eqref{eq:bc}.  
Then there is $\epsilon>0$ such that for any $X'\in B(X,\epsilon) \cap \symMat^\dm_+$, 
we have $h(X')\geq h(X)$ (or $h(X')>h(X)$ when $X' \neq X$, in the case of strict local minimizer).  
Now suppose that a matrix $F\in \symMat^\dm$ satisfies $F^2 = X$. 
Given any $\Delta \in \symMat^d$, we have $(F+\Delta)^2 = FF + \Delta F + F\Delta  + \Delta \Delta$. 
For all $\Delta$ with $\fronorm{\Delta}\leq \min \left\{ \frac{\epsilon}{3\fronorm{F}}, \sqrt{\frac{\epsilon}{3}} \right\}$, we have that  $(F+\Delta)^2 \in B(X,\epsilon) \cap \symMat^\dm_+$. Thus, $g(F+\Delta) = h((F+\Delta)^2)\geq h(F^2)=g(F)$, showing that $F$ is a local minimizer of \eqref{eq:bc.ssv-sym}.

For the case of the {\em strict} local minimizer, we need to consider the possibility that $(F+\Delta)^2=F^2$ for some $\Delta \neq 0$. By considering the eigenvalue decomposition $F = U \Sigma U^\top = \sum_{i=1}^d \sigma_i u_i u_i^T$ (where $\sigma_i$ are the eigenvalues of $F$ and $u_i \in \RR^\dm$ are the eigenvectors), we have that $F^2 =\sum_{i=1}^d \sigma_i^2 u_i u_i^T$. 
Thus we can have $\Delta \neq 0$ with $(F+\Delta)^2 = F^2$ only if there is at least nonzero eigenvalue of $(F+\Delta)$ that is a ``sign flipped" nonzero eigenvalue of $F$.
But (by Weyl's inequality) we can exclude this possibility by choosing $\epsilon'>0$ sufficiently small and requiring $\fronorm{\Delta} \le  \min \left\{ \frac{\epsilon}{3\fronorm{F}}, \sqrt{\frac{\epsilon}{3}}, \epsilon'  \right\}$.

Next, suppose that $F$ is a local minimizer of \eqref{eq:bc.ssv-sym} and satisfies the eigenvalue condition \eqref{eq:ec}. 
Then there is $\epsilon>0$ such that the ball $B(F,\epsilon)$ defined as in \eqref{eq:sb2} is such that   $h(G^2)\geq h(F^2)$ for all $G \in B(F,\epsilon)$, with the equality being strict if $F$ is a strict local minimizer and $G \ne F$.
We define $X=F^2$, and seek $\bar\epsilon>0$ such that $X' \in B(X,\bar\epsilon) \cap \symMat^\dm_+$ has $h(X') \ge h(X)$ (with the inequality being strict when $X' \neq X$).
\textcolor{black}{As we show in Lemma \ref{lem: Xeig_continuity} below, there is a $\bar{\epsilon}>0$ and a positive constant $c$ such that for any matrix $X' \in B(X,\bar\epsilon) \cap \symMat^\dm_+$, we have an eigenvalue decomposition of $F$ as $U \Sigma U^\top$ and matrices $V$ orthogonal, $\Lambda$ nonnegative diagonal such that $X'=V \Lambda V^\top$ with $\fronorm{U-V} \le c \fronorm{X-X'} $ and $\fronorm{\Lambda -\Sigma^2} \le c\fronorm{X-X'}$. 
Moreover, from the same lemma, we have a diagonal matrix  $\bar{\Lambda}$ with $\bar{\Lambda}^2 = \Lambda$ and 
$\fronorm{\bar{\Lambda} - \Sigma}\leq c^{1/2} d^{1/4} \fronorm{X-X'}^{1/2}$.} 
{Thus, defining $G:= V \bar\Lambda V^\top$, we have  $G^2 = X'$ and $\fronorm{F-G} \le c' \max\{\fronorm{X'-X},\fronorm{X'-X}^{\frac{1}{2}}\}$ for some $c'>0$.
We can choose  $\bar\epsilon$ so that for any $X' \in B(X,\bar\epsilon) \cap \symMat^\dm_+$, we have $G \in B(F,\epsilon)$ and thus $h(X') = h(G^2) \ge h(F^2) = h(X)$.}
Thus, we have shown $X$ to be a local minimizer of \eqref{eq:bc}. 

For the claim about {\em strict} local minimizer, we claim that $F$ satisfies the eigenvalue condition \eqref{eq:ec}. Under this claim, note that any $X'$ close to $X$ with $X'\neq X$ gives rise to a factor $G$ constructed as above with $G$ close to $F$ and $X'=G^2$. We have $G \neq F$ since if $G=F$, then $X'=G^2 = X$, contradicting to $X\not=X'$. Thus, we have $g(G)>g(F)$ and $h(X')>h(X)$, as required.

\textcolor{black}{We finally prove that the strict local minimizer $F$ must satisfy the eigenvalue condition \eqref{eq:ec}. Let an eigenvalue decomposition of $F$ be $U\Sigma U^\top$. Suppose that for some eigenvalues $\sigma_i, \sigma_j$ of $F$, we have $0< \sigma_i = -\sigma_j$. Define the matrix $\Delta = U\tilde{\Delta}U^\top$ where 
$\tilde{\Delta}$ is only nonzero in the following entries, for some $\epsilon \in (0,1)$:
\[
[\tilde{\Delta}]_{ii} = -[\tilde{\Delta}]_{jj} = -\sigma_i\epsilon, \quad \tilde{\Delta}_{ij} = \tilde{\Delta}_{ji} = \sigma_i\sqrt{2\epsilon -\epsilon^2}.
\]
Then we find $(F+\Delta)^2 = U (\Sigma+\tilde\Delta)^2 U^\top = U \Sigma^2 U^\top = F^2$ but $F+\Delta \not= F$. 
Hence, $g(F+\Delta) = g(F)$ for all $\Delta$ sufficiently small with the construction outlined here. 
Thus, $F$ cannot be a strict minimizer.
}
%
\end{proof}

\begin{lemma}\label{lem: Xeig_continuity}
\textcolor{black}{Given $F\in \symMat^\dm$ and assume it satisfies the eigenvalue condition \eqref{eq:ec}. Then there is an $\epsilon>0$ and a positive constant $c$ such that 
for $X=F^2$ and any $X' \in B(X,\epsilon) \cap \symMat^\dm_+$, there is an eigenvalue decomposition of $F = \bar{U}\Sigma \bar{U}^\top$ and an eigenvalue decomposition of  $X' = V \Lambda V^\top$ so that the following holds:
\begin{equation}\label{eq: barUVandLSsquare}
\fronorm{\bar{U}-V} \leq c\fronorm{X-X'} \quad \text{and}\quad  
\fronorm{\Lambda - \Sigma^2} \leq c\fronorm{X-X'}.
\end{equation}
Moreover, if we define $\bar{\Lambda} = \mathrm{sign}(\Sigma) \Lambda^{\frac{1}{2}}$, where $\mathrm{sign} (\Sigma)$ is the diagonal matrix whose $(i,i)$ entry is $\mathrm{sign}(\sigma_i)$ \footnote{We define $\mathrm{sign}(0)=1$.}, then we have $\bar{\Lambda} ^2 = \Lambda $ and 
\begin{equation}\label{eq: barLambda}
\fronorm{\bar{\Lambda} - \Sigma}\leq c^{1/2} d^{1/4} \fronorm{X-X'}^{1/2},
\end{equation}
where $c$ is the constant from \eqref{eq: barUVandLSsquare}.
}
\end{lemma}
\begin{proof}
\textcolor{black}{We partition the eigenvectors of $F$ according to the distinct eigenvalues. Formally, let the eigenvalues of $F$ be $\sigma_1\geq \dots \geq \sigma_\dm$ and $K$ be the number of distinct values of $\{\sigma_1,\dots,\sigma_\dm\}$. We partition the set $\{1,\dots, d\}$ according to the distinct values in $\sigma_1,\dots,\sigma_\dm$. Precisely, let $I_1,\dots, I_K$ be a partition of $\{1,\dots,\dm\}$, i.e., $\cup_{i=1}^K I _i = \{1,\dots, \dm\}$ and $I_i \cap I_j =\emptyset$ for any $i\not=j$, such that, for any $I_l$, $l=1,\dots, K$, $\sigma_i = \sigma_j$ if $i,j\in I_l$ and 
for any $1\leq l\not=k\leq K$, 
$\sigma_i \not= \sigma_j$ for $i\in I_l$, $j\in I_k$. 
That is, any two eigenvalues with indices in the same $I_l$ are equal, while those with indices in different $I_l$ are different. 
Let $|I_l|$ be the cardinality of $I_l$. 
Denote an eigenvalue decomposition of $F = U\Sigma U^\top$. 
Then $U\Sigma^2 U^\top$ is an eigenvalue decomposition of $X$. 
Thanks to the eigenvalue condition \eqref{eq:ec}, the partition $I_1,\dots, I_K$ again forms a partition of the eigenvalues $\{\sigma_1^2,\dots,\sigma_\dm^2\}$ of $X$: Eigenvalues of $X$ with indices in the same $I_l$ are equal and those with indices in different $I_l$ are distinct.
Also denote $U_{I_l} \in \mathbb{R}^{\dm \times |I_l|}$ be the matrix whose orthonormal columns span the subspace of eigenvectors corresponding to the eigenvalues in $I_l$, for each $1\leq l\leq K$. 
Note that for each $U_{I_l}$ and any orthonormal matrix $O\in \mathbb{R}^{|I_l|\times |I_l|}$, the columns of $U_{I_l} O$ are also eigenvectors corresponding to the eigenvalues with indices in $I_l$.} 

\textcolor{black}{According to Davis-Kahan \cite[Theorem 2]{yu2015useful}, for each $l=1,\dots,K$, there is some $\epsilon_l>0$, and a positive number $c_l>0$ such that for any $X'\in B(X,\epsilon_i)$ with an eigenvalue decomposition $X' = V\Lambda V^\top$, the following inequality holds for some orthonormal $O_l \in \mathbb{R}^{|I_l|\times |I_l|}$:
\begin{equation}\label{eq: VU_l}
\fronorm{V_{I_l} - U_{I_l}O_l} \leq c_l \fronorm{X'-X}.
\end{equation}
Since the columns of $U_{I_l}O_l$ are still eigenvectors corresponds to eigenvalues in $I_l$, we may collect them together as 
$\bar{U} =[U_{I_1}O_1 \; \cdots U_{I_K}O_K]$
and we still have $F=\bar{U}\Sigma \bar{U}$ and 
$X = \bar{U}\Sigma^2 \bar{U}^\top$. 
Hence, by combining the inequalities \eqref{eq: VU_l} from $l=1$ to $l=K$, we see that there is an $\epsilon_0$ and a positive number $c_0>0$ such that for any $X'\in B(X,\epsilon_0)$ with an eigenvalue decomposition $X' = V\Lambda V^\top$, there is an eigenvalue decomposition of $X=\bar{U}\Sigma \bar{U}^\top$ and 
\begin{equation}\label{eq: UV_XprimeX}
  \fronorm{V-\bar{U}} \leq c_0 \fronorm{X-X'}.  
\end{equation}
The first inequality of \eqref{eq: barUVandLSsquare} now follows from \eqref{eq: UV_XprimeX}, and the second follows from Weyl's inequality by a proper choice of $c$.
}

{For \eqref{eq: barLambda}, we prove that $\fronorm{\bar\Lambda - \Sigma}^2 \le \sqrt{d} \fronorm{\Lambda - \Sigma^2}$ by the following argument.
The result then follows by taking square roots and using the second bound in \eqref{eq: barUVandLSsquare}.
We have
\begin{align*}
\fronorm{\bar\Lambda - \Sigma}^2 &= \sum_{i=1}^d (\bar\Lambda_{ii} - \Sigma_{ii})^2 \\
& = \sum_{i=1}^d \left| |\bar\Lambda_{ii}| - |\Sigma_{ii}| \right|^2 \quad 
\text{since $\bar\Lambda_{ii}$ and $\Sigma_{ii}$  have the same sign} \\
& \le \sum_{i=1}^d \left| |\bar\Lambda_{ii}| - |\Sigma_{ii}| \right| \, (|\bar\Lambda_{ii}| + | \Sigma_{ii}|) \\
& = \sum_{i=1}^d \left| \bar\Lambda_{ii}^2 - \Sigma_{ii}^2 \right|  = \sum_{i=1}^d  \left| \Lambda_{ii} - \Sigma_{ii}^2 \right|  \\
& \le \sqrt{d} \left\{ \sum_{i=1}^d \left| \Lambda_{ii} - \Sigma_{ii}^2 \right|^2 \right\}^{1/2}
= \sqrt{d} \fronorm{\Lambda - \Sigma^2},
\end{align*}
where Cauchy-Schwarz is used for the final inequality.
}
\end{proof}

Next, we illustrate that the eigenvalue condition is actually {\em necessary} for the equivalence of local minimizers. 
That is, there is a function $h:\symMat^\dm\rightarrow \mathbb{R}$ in \eqref{eq:bc} and a matrix $X \in \symMat^\dm_+$ such that $X$ is not a local minimizer of \eqref{eq:bc}, but there is a symmetric $F$ s.t. $X=F^2$ and $F$ is a local minimizer of $g$ in \eqref{eq:bc.ssv-sym}. This example also shows that there are $f$, $C$, $x$, and $F$ of \eqref{eq:f} and \eqref{eq:ssv-sym} such that $x$ is not a local minimizer of \eqref{eq:f} but $(x,F)$ is a local minimizer of \eqref{eq:ssv-sym} under standard reformulation between $\symMat^{n}$ and $\RR^{\frac{n(n+1)}{2}}$.

\begin{exm}
    \label{exm: bc.ssv-sym-local}
Let $d=2$, and define 
\begin{equation} \label{eq:us0}
h(X) := \inprod{A}{X\odot X}+\inprod{B}{X},
\end{equation}
where $\odot$ denotes the Hadamard product and $X \in \symMat^2$.
Note that for any $E \in \symMat^2$, we have
\begin{equation} \label{eq:us1}
\begin{split}
    & h(X+E) - h(X) \\
    & = \inprod{A}{(X+E)\odot (X+E)} - \inprod{A}{X\odot X}  + \inprod{B}{E}  \\
    &=  2 \inprod{A}{X \odot E} + \inprod{A}{E \odot E} + \inprod{B}{E}  \\
    &= \inprod{2A \odot X + B}{E} + \inprod{A}{E \odot E} \\
    &= \inprod{\nabla h(X)}{E} + \inprod{A}{E \odot E}.
\end{split}
\end{equation}
We will show that for certain symmetric matrices $A$ and $B$, $X =I$ satisfies \eqref{eq: bc-foc} but not  \eqref{eq:bc-soc} and is not a local minimizer.  However, there is a symmetric matrix 
\begin{equation} \label{eq:us-1}
F = \begin{bmatrix}
    0 & 1 \\
    1 & 0
\end{bmatrix} \quad\quad \mbox{(note that $F^2 = I$)} 
\end{equation}
that satisfies \eqref{eq: bc.ssv-sym-foc} and \eqref{eq:bc.ssv-sym-soc} of
\eqref{eq:bc.ssv-sym} for $g(F)= h(F^2)$, and is also a local minimizer of \eqref{eq:bc.ssv-sym}.  
These observations demonstrate that a matrix satisfying 2N conditions for  \eqref{eq:bc.ssv-sym} does not necessarily yield a matrix satisfying 2N conditions for  \eqref{eq:bc}. 
Likewise, a local minimizer of \eqref{eq:bc.ssv-sym} may not yield a matrix satisfying 2N conditions for  \eqref{eq:bc} or be a local minimizer of \eqref{eq:bc}. 
%
%
%
%

For $A\in \symMat^2$ in \eqref{eq:us0}, denote $A = \begin{bmatrix}
    a_1 & a_3 \\ 
    a_3 & a_2
\end{bmatrix}$. 
Since $I$ is interior to the cone $\symMat^2_+$, we see that first-order conditions \eqref{eq: bc-foc} are satisfied by $X=I$ if the first-order term in the expansion \eqref{eq:us1} is zero for all $E \in \symMat^2$, that is,
\begin{equation} \label{eq:us2}
0 = \nabla h(I) = 2 A \odot I + B = 2
\begin{bmatrix}
a_1 & 0 \\
0 & a_2
\end{bmatrix} + B.
\end{equation}
The second-order conditions  \eqref{eq:bc-soc}  are {\em not} satisfied at $X=I$ if the second-order term in \eqref{eq:us1} may be negative for some $E \in \symMat^2$, which is equivalent to at least one of $a_1$, $a_2$, or $a_3$ being negative.

Similarly $I$ is {\em not} a local minimizer of \eqref{eq:bc} if \eqref{eq:us2} holds but at least one of $a_1$, $a_2$, or $a_3$ is negative, because if any of $a_1$, $a_2$, or $a_3$ is negative, we can find $E \in \symMat^2$ arbitrarily small such that $h(I+E)<h(I)$ while $I+E$ is feasible for \eqref{eq:bc} provided that $\fronorm{E}\leq 1$.

Consider the following specific choices of $A$ and $B$:
\begin{equation}\label{eq: ap_bc_ab_choice}
A=\begin{bmatrix}
    10 & 5 \\ 
    5 & -1
\end{bmatrix}\quad \text{and}\quad B = \begin{bmatrix}
    -20 & 0 \\
     0 & 2
\end{bmatrix},
\end{equation}
that is, $a_1 = 10$, $a_2=-1$, and $a_3 =5$.
For this choice, it is easy to see that first-order conditions \eqref{eq:us2} are satisfied, but the second-order conditions \eqref{eq:bc-soc}  are not satisfied, nor is $X=I$ a local minimizer: Consider $E = \begin{bmatrix}
    0 & 0 \\
    0 & t
\end{bmatrix}$ for $t\in [0,1]$.
We now proceed to show that {\em $F$ defined by \eqref{eq:us-1} satisfies 2NC and is a local minimizer for \eqref{eq:bc.ssv-sym}.}



From \eqref{eq:us-1}, we have $F^2=I$, so that $g(F) = h(I)$.
The first order condition \eqref{eq: bc.ssv-sym-foc} is $\nabla h(I) \circ F=0$, which indeed holds  because $\nabla h(I) =0$ from \eqref{eq:us2}. 
The second-order condition \eqref{eq:bc.ssv-sym-soc} becomes
\begin{equation}\label{eq: example-soc-strict-local}
D^2 h_X[F\circ \Delta,F\circ \Delta] = 2 \inprod{A}{(F\circ \Delta )\odot(F\circ \Delta )} \geq 0,\quad \text{for all}\quad \Delta \in \symMat^2.
\end{equation}
We introduce notation
\begin{equation} \label{eq:us7}
\Delta = \begin{bmatrix}
    y_1 & y_3 \\
    y_3 & y_2
\end{bmatrix} \quad \implies \quad
F\circ \Delta  = \begin{bmatrix}
    y_3 & \frac{y_1+y_2}{2} \\ 
    \frac{y_1+y_2}{2} & y_3
\end{bmatrix}.
\end{equation}
For $A$ defined in \eqref{eq: ap_bc_ab_choice}, condition \eqref{eq: example-soc-strict-local} is
\begin{equation} \label{eq:us5}
    9 y_3^2 + \frac{5}{2} (y_1+y_2)^2 \ge 0, \quad \mbox{for any $y_1$, $y_2$, $y_3$},
\end{equation}
which is obviously true.
In summary, with $h$ defined by \eqref{eq:us0} and \eqref{eq: ap_bc_ab_choice}, $X=I$ is a 1P but not a 2NP for \eqref{eq:bc}, while for the direct-substitution reformulation \eqref{eq:bc.ssv-sym} with $g(F)=h(F^2)$,  the matrix $F$ defined in \eqref{eq:us-1} (with $F^2=I$) is a 2NP.

Next, we show that $F$ is a {\em local minimizer} of  \eqref{eq:bc.ssv-sym} with $g(F)=h(F^2)$. 
Using \eqref{eq:us1} and  \eqref{eq:us7}, we have the following:
\begin{equation} \label{eq:us8}
\begin{split}
&    g(F+\Delta) - g(F) \\
&= h((F+\Delta)^2) - h(F^2) \\
&= h(I+(2 F \circ \Delta+\Delta^2)) - h(I) \\
&= \langle \nabla h(I) , (2 F \circ \Delta+\Delta^2) \rangle + 
\langle A, (2 F \circ \Delta+\Delta^2) \odot (2 F \circ \Delta+\Delta^2) \rangle \\
&= \left\langle \begin{bmatrix} 10 & 5 \\ 5 & -1 \end{bmatrix}, 
\begin{bmatrix} (2y_3 + y_1^2 + y_3^2)^2  & (y_1+y_2)^2 (1+y_3)^2 \\
(y_1+y_2)^2 (1+y_3)^2  & (2y_3 + y_2^2 + y_3^2)^2  \end{bmatrix} \right\rangle \\
&= 10 (2y_3 + y_1^2 + y_3^2)^2 + 10 (y_1+y_2)^2 (1+y_3)^2 - (2y_3 + y_2^2 + y_3^2)^2,
\end{split}
\end{equation}
where for the second-last equality we used $\nabla h(I)=0$.

We now show that there is $\epsilon>0$ such that for any $\Delta$ with $\fronorm{\Delta} \le \epsilon$, we have that the expression \eqref{eq:us8} is nonnegative. 
We use the notation \eqref{eq:us7} for $\Delta$, denote $z := 2y_3 + y_3^2$ and choose $\epsilon$ small enough that the following bounds hold:
\begin{equation} \label{eq:us10}
(1+y_3)^2 \ge \frac12, \quad (y_1-y_2)^2 \le \frac12.
\end{equation}
We thus have from \eqref{eq:us8} that for $\fronorm{\Delta} \le \epsilon$, 
\begin{equation} \label{eq:us11}
\begin{split}
g(F+\Delta) - g(F) &= 10 (z+ y_1^2 )^2 + 10 (y_1+y_2)^2 (1+y_3)^2 - (z+ y_2^2)^2 \\
& \ge 2 (z+ y_1^2 )^2 + (y_1+y_2)^2 - (z+ y_2^2)^2.
\end{split}
\end{equation}
We apply the identity $2a^2 - b^2 \ge -2(a-b)^2$ to this expression, with $a=z+y_1^2$, $b=z+y_2^2$, to obtain
\begin{align*}
g(F+\Delta) - g(F) & \ge -2(y_1^2-y_2^2) + (y_1+y_2)^2 \\
&= -2(y_1-y_2)^2 (y_1+y_2)^2 + (y_1+y_2)^2 \\
&= (1-2(y_1-y_2)^2)(y_1+y_2)^2 \ge 0,
\end{align*}
where the final inequality follows from \eqref{eq:us10}. Thus $g(F+\Delta) \ge g(F)$ for all $\Delta \in \symMat^2$ with $\fronorm{\Delta} \le \epsilon$, proving that $F$ is a local minimizer for \eqref{eq:bc.ssv-sym}.
\end{exm}

We now discuss the relationship between local minimizers of \eqref{eq:bc} and the nonsymmetric squared-variable parametrization \eqref{eq:bc.dss}.

\begin{theorem}\label{thm:bc-ssv-local_min}
    If $X\in \symMat^\dm$ is a local minimizer of \eqref{eq:bc}, then any $F$ with $X = FF^\top$ is a local minimizer of \eqref{eq:bc.dss}. Conversely, if $F$ is a local minimizer of \eqref{eq:bc.dss}, then $X = FF^\top$ is a local minimizer of \eqref{eq:bc}. 
    In general, the problem \eqref{eq:bc.dss} has no {\em strict} local minimizer (except possibly in the special case in which $F=0$ is a local minimizer).
\end{theorem}
\begin{proof}
    Suppose first that  $X$ is a local minimizer of \eqref{eq:bc}, then there is a ball $B(X,\epsilon)$ defined as in \eqref{eq:sb2} for some $\epsilon>0$ such that for any $X'\in B(X,\epsilon) \cap \symMat^\dm_+$ we have $h(X')\geq h(X)$. 
    Let $F$ be such that $FF^\top = X$. Since $(F+\Delta)(F+\Delta)^\top  = FF^\top + \Delta F^\top + F\Delta ^\top + \Delta \Delta^\top$, we see that  for all $\Delta$ with $\fronorm{\Delta}\leq \min \left\{ \frac{\epsilon}{3\fronorm{F}}, \sqrt{\frac{\epsilon}{3}} \right\}$ we have $(F+\Delta)(F+\Delta)^\top \in B(X,\epsilon)$. 
    Therefore  $g(F+\Delta) = h((F+\Delta)(F+\Delta)^\top)\geq h(FF^\top) = g(F)$, showing that $F$ is a local minimizer. 

    Now suppose that  $F$ is a local minimizer of \eqref{eq:bc.dss}, and set $g^* := g(F) = h(FF^\top)$.
    Then there is a $\epsilon>0$ such that for any $G$ (possibly nonsymmetric) with $\fronorm{F-G} \le \epsilon$, we have $g(G) = h(GG^\top)\geq h(FF^\top) = g^*$.
    Suppose that the SVD of $F$ is $U\Sigma V^\top$, with $U$ and $V$ being orthogonal matrices in $\RR^{d \times d}$ and $\Sigma$ nonnegative diagonal. 
    Note that $\tilde{F}:= FVU^\top= U \Sigma U^\top$ is in $\symMat^d_+$ with $\tilde{F} \tilde{F}^\top = FF^\top$.  
    Given any $\tilde{G} \in \symMat^d$ with $\fronormA{\tilde{F}-\tilde{G}} \le \epsilon$, we have  that $G := \tilde{G} UV^\top$ has $GG^\top = \tilde{G} \tilde{G}^\top$ and $\fronorm{F-G} = \fronormA{\tilde{F}-\tilde{G}} \le \epsilon$, so that $g(\tilde{G}) = g(G) \ge g(F)$.
    We have shown at this point that for the $\epsilon>0$ defined as above, we have for all $\tilde{G} \in B(\tilde{F},\epsilon)$ (with $B$ defined as in \eqref{eq:sb2}; in particular $\tilde{G} \in \symMat^d$) that $g(\tilde{G}) \ge g^*$. Note our $\tilde{F}\succeq 0$ and hence satisfied the eigenvalue condition \eqref{eq:ec}. Thus $\tilde{F}$ is a local minimizer of \eqref{eq:bc.ssv-sym} and satisfies the eigenvalue condition \eqref{eq:ec}. 
    Hence, by Theorem~\ref{thm:bc-ssv-sym-local_min}, we see $X$ is a local minimizer of \eqref{eq:bc}.

To prove the claim about strict local minima of \eqref{eq:bc.dss}, we note that given any local minimizer $F$ of $g$ and any $\epsilon>0$, there is an orthogonal matrix $U \in \RR^{d \times d}$ with $\fronorm{U-I} \le \epsilon/\fronorm{F}$ such that $\fronorm{FU-F} \le \fronorm{F} \fronorm{U-I} \le \epsilon$ such that $g(FU) = h((FU)(FU)^\top) = h(FF^\top) = g(F)$. Thus $F$ cannot be a {\em strict} local minimizer.
\end{proof}

In \cite[Proposition 2.3]{burer2005local}, local minimizers of the linear SDP \eqref{eq: sdp} and the Burer-Monteiro formulation \eqref{eq: sdp.bm} (with $k=\dm$) are shown to be equivalent.

\subsection{Local minimizers of \eqref{eq:f}, \eqref{eq:ssv}, and \eqref{eq:ssv-sym}}

This section discusses relationships between local solutions in the formulations \eqref{eq:f}, \eqref{eq:ssv}, and \eqref{eq:ssv-sym}.
The results and proofs in this section parallel closely those in Appendix~\ref{sec:B1} above.

\begin{theorem}\label{thm:f-ssv-sym-local_min}
    Suppose $x \in \RR^n$ is a (strict) local solution  of \eqref{eq:f}, then any $(x,F) \in \RR^n \times \symMat^d$ with $C(x) = F^2$ is a (strict) local solution of \eqref{eq:ssv-sym}. 
    \textcolor{black}{Conversely, if $(x,F)$ is a local solution of \eqref{eq:ssv-sym} with $F$ satisfying the eigenvalue condition \eqref{eq:ec}, then $x$ is a local solution of \eqref{eq:f}. Moreover, if $(x,F)$ is a strict local solution of \eqref{eq:ssv-sym}, then $x$ is a strict local solution of \eqref{eq:f}.} 
\end{theorem}
\begin{proof}
Suppose first that $x$ is a (strict) local solution of \eqref{eq:f}.
Then there is $\epsilon>0$ such that for all $x' \in \RR^n$ with $\|x-x'\| \le \epsilon$ with $C(x') \succeq 0$, we have $f(x') \ge f(x)$ (with strict inequality in the strict case). 
Defining $F\in \symMat^\dm$ such that  $F^2 = C(x)$, we show that there is $\bar\epsilon>0$ such that for all $(x',F') \in \RR^n \times \symMat^d$ with $\|x-x'\| \le \bar\epsilon$ and $\fronorm{F-F'} \le \bar\epsilon$ with $C(x') = (F')^2$, we have $f(x') \ge f(x)$.
But this claim follows immediately by setting $\bar\epsilon = \epsilon$, since $C(x') = (F')^2 \succeq 0$, we have that $x'$ is feasible for \eqref{eq:f}. Thus $(x',F')$ is a local solution of \eqref{eq:ssv-sym}.

For the strict case, we have that $f(x') > f(x)$ if $x' \neq x$. we need to deal with the case of $x=x'$ but possibly $F' \neq F$. 
Since $(F')^2 = F^2$, we can argue as in the proof of Theorem~\ref{thm:bc-ssv-sym-local_min} that $F$ and $F'$ have common eigenvalues, the only different between them being in the signs of their nonzero eigenvalues. 
As before, we can use Weyl's inequality to reduce $\bar\epsilon$ as needed to ensure that in fact the only matrix $F' \in \symMat^d$ with $(F')^2 = F^2$ and $\fronorm{F-F'} \le \bar\epsilon$ is $F'=F$. 
This completes the proof that $(x,F)$ is a strict local solution of \eqref{eq:ssv-sym}.

Suppose now that $(x,F)$ is a local solution of \eqref{eq:ssv-sym} and $F$ satisfies the eigenvalue condition \eqref{eq:ec}.
Then there exists  $\epsilon>0$ such that for all $(x',G) \in \RR^n \times \symMat^d$ with $\| x-x'\| \le \epsilon$, $\fronorm{F-G} \le \epsilon$, and $C(x') = G^2$, we have $f(x') \ge f(x)$.
We seek a value $\bar\epsilon>0$ such that for all $x' \in \RR^n$ with $\| x-x' \| \le \bar\epsilon$ and $C(x') \succeq 0$, we have $f(x') \ge f(x)$, with (in the strict case) strict inequality for $x \neq x'$.
Defining the eigenvalue decomposition of $F$ to be $F = U \Sigma U^\top$, we have $C(x) = U \Sigma^2 U^\top$. Using Lemma \ref{lem: Xeig_continuity}, 
for any $x' \in \RR^n$ with $\| x-x' \| \le \bar\epsilon$ and $C(x') \succeq 0$, with a possibly redefinition of the eigenvalue decomposition of $F$, we have using smoothness of $C(\cdot)$ that $C(x') = V \Lambda V^\top$ with $\fronorm{U-V} \le c( \| x-x' \|)$ and $\fronorm{\Lambda - \Sigma^2}\le c(\|x-x'\|)$ for an orthogonal matrix $V$ and nonnegative diagonal matrix $\Lambda$, where the function $c:\RR \to \RR_+$ is continuous with $c(0)=0$ and independent of $x'$. 
Defining $\bar\Lambda = \mathrm{sign}(\Sigma) \Lambda^{1/2}$, with $\mathrm{sign}(\Sigma)$ being the diagonal matrix whose $(i,i)$ entry is $\mathrm{sign} (\sigma_i)$, we have that $\bar\Lambda^2 = \Lambda$ and $\fronorm{\bar\Lambda - \Sigma} \le c(\|x-x'\|)$, with a possible redefinition of $c$ (but retaining its salient properties) using Lemma \ref{lem: Xeig_continuity}.  
Defining $G:= V \bar\Lambda V^\top$, we have  $G^2 = C(x')$ and $\fronorm{F-G} \le c(\|x-x'\|)$, with another possible redefinition of $c$.
Thus we can choose $\bar\epsilon$ sufficiently small that for all $x' \in \RR^n$ with $\| x-x' \| \le \bar\epsilon$ and $C(x') \succeq 0$, we have that $(x',G) \in \RR^n \times \symMat^d$ with $\|x-x'\| \le \epsilon$, $\fronorm{F-G} \le \epsilon$, and  $C(x') = G^2$, and hence $f(x') \ge f(x)$, as required.

For the claim about {\em strict} local minimizer, note that any $x'$ close to $x$ with $x'\neq x$ and $C(x') \succeq 0$ gives rise to a factor $G$ constructed as above with $G$ close to $F$ and $C(x') = G^2$ so long as $F$ satisfies the eigenvalue condition \eqref{eq:ec}. But the condition follows from the last part of the proof of Theorem \ref{thm:bc-ssv-local_min}. 
Using $(x,F)$ being a strict local minimizer, we see $(x,F)\not=(x',G)$ implies that $f(x') > f(x)$. Our proof is complete.
\end{proof}

\begin{theorem}\label{thm:f-ssv-local_min}
    Suppose $x \in \RR^n$ is a local solution of \eqref{eq:f}, then any $(x,F) \in \RR^n \times \RR^{d \times d}$ with $C(x) = FF^\top$ is a local solution of \eqref{eq:ssv}. Conversely, if $(x,F)$ is a local solution of \eqref{eq:ssv}, then $x$ is a local solution of \eqref{eq:f}. 
    In general, the problem \eqref{eq:ssv} has no {\em strict} local solution (except possibly in the special case in which $C(x)=0$ and $F=0$).
\end{theorem}
\begin{proof}
    Suppose first that $x$ is a local solution of \eqref{eq:f}.
    Then there is $\epsilon>0$ such that any $x' \in \RR^n$ with $\|x-x'\| \le \epsilon$ and $C(x') \succeq 0$ has  $f(x')\geq f(x)$. 
    Defining any $F$ such that $C(x) = FF^\top$, we show that $(x,F)$ is a local solution of \eqref{eq:ssv} by finding $\bar\epsilon >0$ such that for all $(x',G)$ with $\|x-x'\| \le \bar\epsilon$ and $\fronorm{F-G} \le \bar\epsilon$, we have $f(x') \ge f(x)$.
    The claim follows immediately when we set $\bar\epsilon = \epsilon$

Suppose next that $(x,F)$ is a local solution of \eqref{eq:ssv}, and set $f^* := f(x)$.
Then there is $\epsilon>0$ such that for all $(x',G)$ with $\|x-x'\| \le \epsilon$ and $\fronorm{F-G} \le \epsilon$, we have $f(x') \ge f^* = f(x)$.
Suppose that the SVD of $F$ is $U \Sigma V^\top$, with $U$ and $V$ being orthogonal matrices in $\RR^{d \times d}$ and $\Sigma$ nonnegative diagonal.
Then $\tilde{F} := F V U^\top = U \Sigma U^T$ is in $\symMat^d_+$ with $\tilde{F} \tilde{F}^\top = FF^\top$.
Given any $(x',\tilde{G}) \in \RR^n \times \symMat^d$ with $\|x-x'\|\le\epsilon$, $\fronormA{\tilde{F}-\tilde{G}} \le \epsilon$, and $C(x') = \tilde{G} \tilde{G}^\top \succeq 0$,  we define $G := \tilde{G} UV^\top$. 
We thus have $C(x') = \tilde{G} \tilde{G}^\top = GG^\top$ and $\fronorm{F-G} = \fronormA{\tilde{F}-\tilde{G}} \le \epsilon$. 
Thus by the assumption that $(x,F)$ is a local solution of \eqref{eq:ssv} and for $\epsilon$ defined as above, we have $f(x') \ge f^*$.
We have shown at this point that for $\epsilon$ defined as above, we have for $(x',\tilde{G}) \in \RR^n \times \symMat^d$ with $\|x-x'\| \le \epsilon$, $\fronormA{\tilde{F}-\tilde{G}} \le \epsilon$, and $C(x') = \tilde{G}\tilde{G}^\top$ we have $f(x') \ge f^*$ and $\tilde{F}$ satisfies the eigenvalue condition \eqref{eq:ec}.
Hence, by Theorem~\ref{thm:f-ssv-sym-local_min}, we see that $x$ is a local minimizer of \eqref{eq:f}.

The claim about (lack of) strict local solutions of \eqref{eq:ssv} follows as in the last paragraph of Theorem~\ref{thm:bc-ssv-local_min}.
\end{proof}


\section{Stronger second-order necessary conditions} \label{sec: SOC discussion}
\label{app:C}

We discuss here a stronger form of 2NC for  \eqref{eq:f}, termed s2NC, considered in \cite{shapiro1997first} and \cite{lourencco2018optimality}. 
We only consider the general case \eqref{eq:f} in this section and show that it is the same as the 2NC under a strict complementarity condition. 

\begin{definition}[Stronger second order necessary conditions of \eqref{eq:f}]\label{def:f.s2NC}
We say that a point $(x,\Lambda) \in \real^n \times \symMat^\dm $ satisfies {\em stronger second-order necessary} conditions (s2NC) to be a solution of  \eqref{eq:f} if it satisfies first-order conditions \eqref{eq: f-foc} and, defining $r := \rank(C(x))$ and  $V \in \R^{\dm \times (d-r)}$ to be the matrix with  orthonormal columns whose column space is the null space of $C(x)$, we also have
\begin{equation}\label{eq: f-ssoc}\tag{NSDP-s2NC}
D^2 \cL(x)[z,z] + 2 \tr \left( DC_x[z] C^\dagger (x) DC_x[z] \Lambda \right)\geq 0,
\end{equation}
for all $z\in \RR^n$ such that $V^\top  D C_x[z]V\succeq 0$ and $\tr(DC_x[z] \Lambda) =0$.
Here $C(x)^\dagger$ is the pseudo-inverse of the matrix $C(x)$.
\end{definition}

Comparing Definitions~\ref{def:f} and \ref{def:f.s2NC}, we see that the only difference is in the range of $z$. 
In Definition~\ref{def:f} we need  $z$ to satisfy $V^\top  D C_x[z]V = 0$, thus requiring the inequality \eqref{eq: f-soc} to hold over a \emph{subspace} of $\real^n$. 
In Definition~\ref{def:f.s2NC}, we require $z$ to satisfy $V^\top  D C_x[z]V\succeq 0$ and $\tr(DC_x[z] \Lambda) =0$, thus requiring \eqref{eq: f-ssoc} to hold over a \emph{cone}. 
It may not be immediately clear why the cone of Definition~\ref{def:f.s2NC} contains the subspace of Definition~\ref{def:f}, but we argue next that such is indeed the case. 

Note that from the first order condition \eqref{eq: f-foc}, we know $\Lambda C(x) =C(x) \Lambda = 0$ and $\Lambda \succeq 0$. 
Hence,  the range of $\Lambda$ is contained in the null space of $C(x)$.
WLOG, we can choose the matrix $V$ specified in Definitions~\ref{def:f} and \ref{def:f.s2NC} to have the structure $V = [ V_1 \;V_2]$ where $V_1$ spans the range space of $\Lambda$ and in fact that $V_1^\top \Lambda V_1 \succ 0$ (positive definite). 
Suppose $V_1 \in \mathbb{R}^{d\times {s}}$ and $V_2 \in \mathbb{R}^{d\times (d-r-s)}$. 
Let $U\in \mathbb{R}^{d\times r}$ be a matrix with orthonormal columns that span the orthogonal complement of the null space of $C(x)$. 
Then $H = [V \; U ] = [V_1\;V_2\; U ]$ is an orthogonal $d \times d$ matrix. 
We thus have
\[
H^\top DC_x[z] H = 
\begin{bmatrix}
V^\top DC_x[z] V & V^\top DC_x[z] U \\ 
U^\top DC_x[z] V & U^\top DC_x[z] U
\end{bmatrix}, \quad H^\top \Lambda H = 
\begin{bmatrix}
    V^\top \Lambda V & 0 \\
    0 & 0
\end{bmatrix}.
\]
Thus, under \eqref{eq: f-foc}, the set of conditions $V^\top  D C_x[z]V\succeq 0$ and 
\[
0=\tr(DC_x[z] \Lambda) = 
\tr( (H^\top DC_x[z]H) (H^\top \Lambda H))=\tr( (V^\top DC_x[z]V) (V^\top \Lambda V))
\]
is equivalent to 
\begin{equation}\label{eq: range_of_z_V}
V^\top  D C_x[z]V\succeq 0
\quad \text{and} \quad (V^\top DC_x[z]V) (V^\top \Lambda V)=0.
\end{equation} 
Breaking down the upper left blocks of the matrices above further, we have 
\begin{equation}\label{eq: z_range_V}
V^\top DC_x[z] V = 
\begin{bmatrix}
V_1^\top DC_x[z] V_1 & V_1^\top DC_x[z] V_2 \\ V_2^\top DC_x[z] V_1 & V_2^\top DC_x[z] V_2
\end{bmatrix}, \quad V^\top \Lambda V = 
\begin{bmatrix}
    V_1^\top \Lambda V_1  & 0 \\
    0 & 0
\end{bmatrix},
\end{equation}
where $V_1^\top \Lambda V_1$ is symmetric positive definite, as noted above.
From this decomposition, the condition
$(V^\top DC_x[z]V) (V^\top \Lambda V)=0$ is equivalent to  
\[
0= 
\begin{bmatrix}
    V_1^\top D C_x[z] V_1 \\
    V_2^\top D C_x[z] V_1 
\end{bmatrix} = V^\top D C_x[z] V_1.
\]
By combining this equality with $V^\top  D C_x[z]V\succeq 0$, we see that the conditions \eqref{eq: range_of_z_V} are further equivalent to
\begin{equation}\label{eq: range_of_z}
    V^\top D C_x[z] V_1 = 0 \quad \text{and}\quad V_2 ^\top DC_x[z] V_2 \succeq 0.
\end{equation}
Recall that the vectors $z \in \real^n$ in Definition \ref{def:f} are required to satisfy  $V^\top DC_x[z]V =0$, which using 
\eqref{eq: z_range_V} is equivalent to 
\begin{equation}\label{eq: range_of_z_w_soc}
    V^\top D C_x[z] V_1 = 0, \quad \text{and}\quad V_2 ^\top DC_x[z] V_2 = 0.
\end{equation}
We see that \eqref{eq: range_of_z_w_soc} is a more stringent condition on $z$ than in \eqref{eq: range_of_z}, and in fact the cone of $z$ values described in Definition~\ref{def:f.s2NC} contains the subspace described in Definition~\ref{def:f}.

Next, we argue that s2NC reduces to 2NC under strict complementarity.
The strict complementarity condition \cite[(18)]{shapiro1997first} is that
\begin{equation}\label{eq: strict complementarity}
    \rank(\Lambda) + \rank(C(x)) = \dm.
\end{equation}
Since $\Lambda C(x) =0$, using the rank-nullity theorem, we know strict complementarity implies that  
\[
\range(\Lambda) = \nullspace(C(x)) \quad \text{and}\quad V_1 = V.
\]
Hence, we see that \eqref{eq: range_of_z} and \eqref{eq: range_of_z_w_soc} both reduce to $V^\top DC_x[z]V =0$, so the set of values of $z$ for which the relevant curvature condition needs to hold is the same in both Definition~\ref{def:f} and Definition~\ref{def:f.s2NC}

\section{Equivalence of 1Ps and global minimizers for \eqref{eq: nnm} and \eqref{eq: PSDreg}}
\label{app:D}

\begin{lemma}\label{lem: eq1Pnnm}
If $X$ is a 1P of \eqref{eq: nnm}, then \eqref{eq: PSDreg} admits a 1P $\bar{X}$  whose off-diagonal is $X$. Conversely,  if $\bar{X}$ is a 1P  of \eqref{eq: PSDreg}, then its off-diagonal submatrix $X$ from \eqref{eq: PSDreg}
is a 1P of \eqref{eq: nnm}. 

Moreover, all previous statements continue to hold if we replace ``1P" by ``global solution."
\end{lemma}
\begin{proof}
We assume without loss of generality that $\lambda =1$. 

Let $r$ be the rank of $X$ and write the SVD of $X$ as $X = U\Sigma V^\top $ where $U\in \mathbb{R}^{\dm_1\times r}$ and $V\in \mathbb{R}^{\dm_2\times r}$ are orthonormal and $\Sigma \in \mathbb{R}^{r\times r}$ is a diagonal matrix. 
Recall that a matrix $X$ is a 1P of   \eqref{eq: nnm} if 
\begin{equation}\label{eq: nnm_1P}
    -\nabla h(X) \in \partial \nucnorm{X} = \{ H \mid U^\top H = V^\top \;,\; 
     V^\top H^\top  = U^\top ,\;\text{and}\;
    \opnorm{H}\leq 1\},
\end{equation}
where $\opnorm{H}$ is the largest singular value of $H$.
(This description of subdifferential $\partial \nucnorm{X}$ can be found in 
\cite[Example 2]{watson1992characterization}.) Hence $X$ is a 1P of \eqref{eq: nnm} if 
\begin{equation} \label{eq:sj7}
    \opnorm{\nabla h(X)}\leq 1, \;\;
U^\top \nabla h(X) = -V^\top, \;\; V^\top \nabla h(X)^\top = -U^\top.
\end{equation}
From \eqref{eq: bc-foc}, the first-order necessary conditions of a matrix $\bar{X}$ for \eqref{eq: PSDreg} are
\begin{equation}\label{eq: 1Nnnm}
    \bar{X}\succeq 0, \;\;
    \nabla \bar{h}(\bar{X})\succeq 0, \;\; \nabla \bar{h}(\bar{X}) \bar{X} =0.
\end{equation}
Recall $\bar{h}(\bar{X}) = h(X) + \frac{1}{2}\tr(\bar{X})$.

We now prove the first claim of the lemma:  If $X$ is a 1P of \eqref{eq: nnm}, we construct a 1P $\bar{X}$ for \eqref{eq: PSDreg} whose off-diagonal is $X$. 
Next , we prove the reverse direction. 
Finally,  we prove the corresponding claims about global solutions.

\noindent \underline{\textbf{\eqref{eq: nnm_1P} yields \eqref{eq: 1Nnnm}.}}
Given $X = U \Sigma V^T$ satisfying \eqref{eq: nnm_1P}, we show that the following matrix $\bar{X}$ satisfies \eqref{eq: 1Nnnm}:
\begin{equation}\label{eq: bar_X_from_X} 
\bar{X} = 
\begin{bmatrix}
    U\Sigma U^\top & X \\ 
    X^\top & V\Sigma V^\top
\end{bmatrix} = 
\left( \frac{1}{\sqrt{2}} 
\begin{bmatrix}
    U \\ V
\end{bmatrix} \right) (2\Sigma) \left(\frac{1}{\sqrt{2}}\begin{bmatrix}
    U \\ V
\end{bmatrix}\right)^\top.
\end{equation}
It follows immediately from this definition that  $\bar{X}\succeq 0$. 
Note that 
\begin{equation}\label{eq: trace_bar_X_nucnorm_X}
\tr(\bar{X}) = 2\tr(\Sigma) = 2\nucnorm{X}
\end{equation}
To see $\nabla \bar{h}(\bar{X}) \succeq 0$, we use the Schur complement characterization:  
\begin{equation}\label{eq: nabla_h_bar}
\nabla \bar{h}(\bar{X}) = \frac{1}{2}
\begin{bmatrix}
   I_{\dm_1}  & \nabla h(X)\\ 
    \nabla h(X)^\top & I_{\dm_2} 
\end{bmatrix}\succeq 0
\iff I_{d_2} -\nabla h(X)^\top \nabla h(X)\succeq 0. 
\end{equation} 
But the right-hand inequality holds because  $\opnorm{\nabla h(X)}\leq 1$ from \eqref{eq:sj7}. 
It  remains to show that $\nabla \bar{h}(\bar{X}) \bar{X}=0$. 
From \eqref{eq: bar_X_from_X} and \eqref{eq: nabla_h_bar}, we have
\begin{align}
\nonumber 
    \nabla \bar{h} (\bar{X})\bar{X} & = \frac12 \begin{bmatrix}
   I_{\dm_1}  & \nabla h(X)\\ 
    \nabla h(X)^\top & I_{\dm_2} 
\end{bmatrix}
\begin{bmatrix}
    U\Sigma U^\top & X \\ 
    X^\top & V\Sigma V^\top
\end{bmatrix} \\
\label{eq:sj9}
& =
\frac12 \begin{bmatrix}
    U \Sigma U^\top + \nabla h(X) X^\top & X + \nabla h(X) V \Sigma V^\top  \\  \nabla h(X)^\top U \Sigma U^\top +X^\top &  \nabla h(X)^\top X + V \Sigma V^\top 
\end{bmatrix}.
\end{align}
For the top left block, we have from $X = U \Sigma V^\top$ and \eqref{eq:sj7} that
\[
U \Sigma U^\top + \nabla h(X) X^\top = U \Sigma U^\top +(\nabla h(X) V) \Sigma  U^\top = U \Sigma U^\top  - U \Sigma U^\top  = 0.
\]
We can show by similar arguments that the other three blocks in \eqref{eq:sj9} are also zero matrices, completing the proof.

\smallskip

\noindent \underline{\textbf{\eqref{eq: 1Nnnm} yields \eqref{eq: nnm_1P}.}}
Supposing  now that $\bar{X}$ is a 1P of \eqref{eq: PSDreg}, we show that its off-diagonal submatrix $X$  is a 1P of \eqref{eq: nnm}. 
It suffices to show that when \eqref{eq: 1Nnnm}  holds then \eqref{eq:sj7} holds, where $U$ and $V$ are the singular vectors of $X$.
Note first that the condition $\opnorm{\nabla {h}(X)}\leq 1$ is satisfied because of $\nabla {\bar{h}}(\bar{X}) \succeq 0$ and \eqref{eq: nabla_h_bar}. We are left to prove that the other two conditions hold. 

Given the symmetric matrix $\bar{X} \succeq 0$ satisfying \eqref{eq: 1Nnnm}, with rank $\bar{r}$. Note that $\bar{r}\leq \min\{d_1,d_2\}$. To see this,  note from \eqref{eq: nabla_h_bar}, we know the rank of $\nabla \bar{h}(\bar{X})$ is at least $\max\{d_1,d_2\}$ by inspecting the number of independent columns of $\nabla \bar{h}(\bar{X})$. Combining this fact with $\nabla \bar{h}(\bar{X}) \bar{X} =0$ in \eqref{eq: 1Nnnm}, we see $\bar{r}\leq \min\{d_1,d_2\}$.
We write the eigenvalue decomposition of $\bar{X}$ in terms of a matrix
$Z = 
\begin{bmatrix}
    \bar{U} \\ \bar{V} 
\end{bmatrix} \in \mathbb{R}^{(\dm_1+\dm_2)\times \bar{r}}$ with orthonormal columns as follows:  
\begin{equation}\label{eq: XWbarUbarV}
\begin{bmatrix}
         W_1 & X \\
         X^\top & W_2
     \end{bmatrix} = 
     \bar{X} = 
     Z \bar{\Sigma} Z^\top =
\begin{bmatrix}
    \bar{U}\bar{\Sigma} \bar{U}^\top & \bar{U} \bar{\Sigma} \bar{V}^\top\\
    \bar{V} \bar{\Sigma} \bar{U}^\top & 
    \bar{V} \bar{\Sigma} \bar{V}^\top
\end{bmatrix} =
\begin{bmatrix}
\bar{U} \\ \bar{V}
\end{bmatrix}
\bar{\Sigma}
\begin{bmatrix}
\bar{U} \\ \bar{V}
\end{bmatrix}^\top,
\end{equation}
where $\bar\Sigma$ is the positive diagonal matrix of eigenvalues (with the diagonals arranged in nonincreasing order).
Since $\nabla \bar{h}(\bar{X}) \bar{X} =0$, we have from \eqref{eq: XWbarUbarV} and the form of $\nabla \bar{h} (\bar{X})$ in \eqref{eq: nabla_h_bar} that 
\begin{equation}\label{eq: nabla _bar_h_bar_X_W}
0 = \nabla \bar{h}(\bar{X}) Z = \frac{1}{2}
\begin{bmatrix}
    \bar{U} + \nabla h(X) \bar{V} \\ 
    \nabla h(X)^\top \bar{U}  + \bar{V}
\end{bmatrix},
\end{equation}
so that 
\begin{equation}
    \label{eq:sj8}
    \bar{U}^\top \nabla h(X) \bar{V} = -\bar{U}^\top \bar{U} = -\bar{V}^\top \bar{V}.
\end{equation}
We show now that by setting
\[
U := \sqrt{2} \bar{U}, \;\; V := \sqrt{2} \bar{V}, \;\; \Sigma := \frac12 \bar\Sigma,
\]
we have that $X=U \Sigma V^T$ is the SVD of $X$, where in particular $U$ and $V$ are orthonormal, and also that $U$ and $V$ satisfy the second and third conditions of \eqref{eq:sj7}.
Because of \eqref{eq:sj8}, together with $I_{\bar{r}} = Z^\top Z = \bar{U}^\top \bar{U} +  \bar{V}^\top \bar{V}$, we have that $\bar{U}^\top \bar{U} =  \bar{V}^\top \bar{V} = \frac12 I_{\bar{r}}$, so $U$ and $V$ have orthonormal columns. 
Meanwhile, we have from \eqref{eq: XWbarUbarV} that  $X = (\sqrt{2}\bar{U}) \frac{\bar{\Sigma}}{2} (\sqrt{2}\bar{V})^\top = U \Sigma V^\top$, where $\Sigma$ is positive diagonal with nonincreasing diagonals. 
Finally, by multiplying \eqref{eq: nabla _bar_h_bar_X_W} by $2\sqrt{2}$, substituting the definitions $U = \sqrt{2} \bar{U}$ and $V = \sqrt{2} \bar{V}$, we obtain the second and third conditions of \eqref{eq:sj7}.


Note that \eqref{eq: trace_bar_X_nucnorm_X} holds because from \eqref{eq: XWbarUbarV} and $\Sigma = \tfrac12 \bar\Sigma$, we have $\tr (\bar{X}) = \tr (\bar\Sigma) = 2 \tr (\Sigma)$, while $\tr (\Sigma) = \| X \|_*$.

\smallskip

\noindent{\underline{\textbf{Equivalence between global solutions of \eqref{eq: nnm} and \eqref{eq: PSDreg}.}}} 
If $X$ is a global solution of \eqref{eq: nnm}, we construct $\bar{X}$ as in \eqref{eq: bar_X_from_X}. 
(Note that \eqref{eq: trace_bar_X_nucnorm_X} holds.)
Recall the definition \eqref{eq: nucnorm_SDP} of $\nucnorm{X}$.
For any $\bar{X}' = 
\begin{bmatrix}
    W_1' &X' \\ 
    X'^\top & W_2'
\end{bmatrix}\succeq 0$, feasible for \eqref{eq: PSDreg}, we have 
\begin{equation*}
    \bar{h}(\bar{X}') 
    = h(X') + \frac{1}{2}\left(\tr(W_1')+\tr(W_2')\right)
    \overset{(a)}{\geq} h(X') + \nucnorm{X'} 
   \overset{(b)}{\geq} h(X) + \nucnorm{X} 
   \overset{(c)}{=} \bar{h}(\bar{X}).
\end{equation*}
Here, step $(a)$ is due to \eqref{eq: nucnorm_SDP}; the step $(b)$ is due to the optimality of $X$ in \eqref{eq: nnm}, with $\lambda=1$; and  step $(c)$ is due to \eqref{eq: trace_bar_X_nucnorm_X}, along with the definition of $\bar{h}(\cdot)$ from \eqref{eq: PSDreg}. 
Thus, $\bar{X}$ is a global solution of \eqref{eq: PSDreg}.

For the converse, suppose that  $\bar{X} = \begin{bmatrix}
    W_1 & X \\ 
    X^\top & W_2
\end{bmatrix}$ is a global solution of \eqref{eq: PSDreg}.
Take any $X'\in \mathbb{R}^{\dm_1 \times \dm_2}$, define $W_1'$ and $W_2'$ to be the solutions of \eqref{eq: nucnorm_SDP} corresponding to $X'$, and set $\bar{X}' := 
\begin{bmatrix}
    W_1' & X' \\
    X'^\top & W_2'
\end{bmatrix} \succeq 0$.
Then we have
\[
h(X') + \nucnorm{X'}
\overset{(a)}{=} h(X') + \tfrac{1}{2} \tr(\bar{X}') 
\overset{(b)}{\geq }\bar{h}(\bar{X}) 
\overset{(c)}{\geq} h(X) + \nucnorm{X}.
\]
Step $(a)$ uses the definition of $\bar{X}'$ from \eqref{eq: nucnorm_SDP}  and step $(b)$ uses optimality of $\bar{X}$ and the definition of $\bar{h}(\cdot)$ from \eqref{eq: PSDreg}. The last step $(c)$ uses that $\bar{h}(\bar{X}) = h(X) + \frac{1}{2}(\tr(W_1)+\tr(W_2))\geq h(X) + \nucnorm{X}$ where the inequality is due to \eqref{eq: nucnorm_SDP}.
\end{proof}

\end{document}